\documentclass[11pt,leqno]{article}%
\usepackage{amssymb,bbm,amsmath,amsfonts,amsthm,
amscd}
\usepackage{vmargin}
\usepackage{asymptote}
\usepackage[dvipsnames]{xcolor}
\usepackage{
 array}
 
\usepackage{xspace}
 
 \usepackage{mathrsfs}
\usepackage{macros}

\usepackage[toc,page]{appendix}

\usepackage{enumerate}

\usepackage{color}

\usepackage{xcolor}

\usepackage[linktocpage,colorlinks=true]{hyperref}

\hypersetup{urlcolor=blue, citecolor=red, linkcolor=blue}

\makeatletter 

\@addtoreset{equation}{section}
\renewcommand{\theequation}{\arabic{section}.\arabic{equation}}

\begin{document}

\title{Navier-Stokes equations in the whole space with an eddy viscosity}

\author{ Roger Lewandowski\thanks{IRMAR, UMR 6625, Universit\'e Rennes 1,
   Campus Beaulieu, 35042 Rennes cedex FRANCE;
   Roger.Lewandowski@univ-rennes1.fr,
   http://perso.univ-rennes1.fr/roger.lewandowski/}} \date{}
\maketitle

\begin{abstract} We study the Navier-Stokes equations with an extra eddy viscosity term in the whole space $\R^3$. We introduce a suitable regularized system  for which we prove the existence of a regular solution defined for all time. 
We prove that when the regularizing parameter goes to zero, the solution of the regularized system converges to a turbulent solution of  the initial system. 
\end{abstract}
MCS Classification: 35Q30, 35D30, 76D03, 76D05.
\smallskip


\begin{center} {\sl In memory of Jean Leray} \end{center} 

\section{Introduction}

Let us consider the incompressible Navier-Stokes Equations (NSE) in the whole space $\R^3$ with an extra eddy viscosity term:
\BEQ \label{eq:NSE} \left \{ \begin{array}{ll} 
\p_t \uv + \uv \cdot \g \uv - \nu \Delta \uv - \div( A \g \uv) + \g p = 0, & \hskip 3cm\hbox{(i)} \\ 
\div \, \uv = 0,   \phantom{\int_0^1} & \hskip 3cm \hbox{(ii)} 
\end{array} \right. 
\EEQ
where $\uv = \uv (t, \x) = (u_1(t, \x), u_2(t, \x), u_3(t, \x))$ denotes the velocity, $p = p(t, \x)$ the pressure, $t \ge 0$, $\x \in \R^3$, $\nu>0$ is the kinetic viscosity and $A = A(t, \x)$ is an eddy viscosity. 

This PDE system arises from turbulence modeling, the purpose of which is the calculation  of averaged or filtered  fields associated to a given turbulent flow. 
Eddy viscosities  are usually introduced to model the Reynolds stress of such flows, according to the Boussinesq assumption (see for instance \cite{CL14, SP00}).

This system was already studied in the case of bounded domains with various boundary conditions (see in \cite{CL14}, chapters 6 to 8 for a comprehensive presentation). However, so far we know, it has never been investigated before in the case of the whole space, which motivates the present study.

We prove in this paper the existence of a turbulent solution to the NSE (\ref{eq:NSE}), global in time,  through a suitable variational formulation on
the basis of the assumptions: 
\begin{enumerate}[i)]
\item $\uv_0 \in L^2(\R^3)^3$, $\div \, \uv_0 = 0$, 
\item $A \ge 0$, 
$A \in C_b(\R_+, W^{1, \infty}(\R^3))$ and is of compact support in $\R^3$ uniformly in time. 
\end{enumerate}
One of the key features  of this solution is that it satisfies  an energy inequality. The notion of turbulent solution was initially introduced by J. Leray \cite{Ler1934} when $A=0$, what makes our result a generalization of Leray's result. However, one part of Leray's program does not directly apply to the case $A \not=0$, since the eddy viscosity term brings unexpected issues. Things must be reconsidered, which motivates our developments that do not appear in 
Leray's paper (see  subsection \ref{sec:additional_comments}  for further explanations).

The turbulent solution is constructed as a limit of regular solutions when $\E \to 0$ of   the regularized NSE,
\BEQ \label{eq:NSE_eps_intro} \left \{ \begin{array}{l} 
\p_t \uv + \um \cdot \g \uv - \nu \Delta \uv - \overline{ \div \, ( A \g \um)} + \g p = 0,  \\
\div \, \uv = 0, \phantom{\int }\\
\uv_{t=0} = \overline {\uv_0},
\end{array} \right. 
\EEQ
 where $\overline \psi = \rho_\E \star \psi$ for a given mollifier $\rho$, and $\rho_\E = \E^{-3} \rho(\x / \E)$. The regularized convection term $\um \cdot \g \uv$ was initially introduced by  J. Leray. We introduce the regularized eddy viscosity term
$-\overline{ \div \, ( A \g \um)}$ in order to preserve the dissipation feature of the eddy viscosity. 
 
 A large part of the paper is devoted to study the regularized NSE   (\ref{eq:NSE_eps_intro}). The Oseen representation  formula \cite{CO11, CWO27} combined with a fixed point procedure yields the existence of a unique regular solution, global in time,  which means 
 a solution of class $C^\infty$  in time and space defined for $t \in [0, \infty[$, with no singularity, the  $H^m$  norms of which are driven by  the $L^2$ norm of $\uv_0$, 
 $\E$, the shape of $\rho$ and its derivatives. This solution satisfies the energy balance, which provides valuable
estimates that do not depend on $\E$. 

We then show that the solution of the regularized NSE (\ref{eq:NSE_eps_intro}) converges to a turbulent solution of the NSE (\ref{eq:NSE}) when $\E \to 0$. The proof  makes
use of an estimate of the solution of  (\ref{eq:NSE_eps_intro}) for large values of $| \x |$,   uniform in $\E$, which allows to apply standard compactness arguments on bounded domains. The assumption ``$A$ is of compact support uniformly in $t$" plays a role at this stage, and it is likely that it
could be replaced by a suitable decay assumption of $A$ for large values of $| \x |$. 

The paper is organized as follows. We first define the notions of regular and turbulent solutions. Afterwards, we focus on the approximated system (\ref{eq:NSE_eps_intro}). In particular, we get {\sl \`a priori} estimates of all $H^m$ norms of the velocity  for any {\sl \`a priori } regular solutions, estimates derived from the Oseen potential and  non linear Volterra equations. Then we prove the existence of  a unique regular solution. Afterwards, we pass to the limit in the equations to construct the turbulent solution. 

We conclude the paper by a series of remarks and additional open problems. We explain why Leray's frame cannot be reproduced turnkey when $A \not=0$. We also make natural connections between the present work  and models such as 
Bardina (see Layton-Lewandowski \cite{LL03, LL06},  Cao-Lunasin-Titi \cite{CLT06}) and Leray-$\alpha$ (see Cheskidov-Holm-Olson-Titi \cite{CHOTI05}). 

Finally, The appendice \ref{app:ossen} aims to prove the basic estimates about the Oseen's tensor. The appendice \ref{ap:appen_volt} is devoted to the non linear Volterra equations and what we have called the 
V-maximum principle, intensively  used to get estimates.  
\medskip

{\sl Acknowledgement}. I thank my colleagues Christophe Cheverry and Taoufik Hmidi of the Institute of Mathematical Reasearch of Rennes (IRMAR) for many 
stimulating discussions about the Navier-Stokes equations. 

I am gratefull to Paul Alphonse and Adrien Laurent, students at the Ecole Normale Sup\'erieure de Rennes, who have attended my master course about the Navier-Stokes Equations  at  the University of Rennes 1, during the academic year 
2016-2017. Their remarks and comments have been very useful to improve my course, and they have entirely written the Appendice \ref{app:ossen} included in this paper. 

Finally, I warmly thank Luc Tartar  for the very smart proof of the V-maximum principle reported in Appendice \ref{ap:appen_volt}, and more generally for many valuable discussions on Leray's 1934 paper and  fluid mechanics over the past years.

\section{Regular and turbulent solutions} 

\subsection{Regular solutions} 
Let $\alpha = (\alpha_1, \alpha_2, \alpha_3) \in \N^3$, $| \alpha | = \alpha_1 + \alpha_2 + \alpha_3$, 
$$ D^\alpha \uv = (D^\alpha u_1, D^\alpha u_2, D^\alpha u_3), \quad D^\alpha u_i = { \p ^{| \alpha | }u_i \over \p x_1 ^{\alpha_1} \p x_2 ^{\alpha_2 } \p x_3 ^{\alpha_3} }.$$
For any given $m \in \N$, when we write $D^m \uv$ we assume that $D^\alpha \uv$ is well defined whatever $\alpha$ such that $| \alpha | = m$, and in practical calculations 
$$ | D^m \uv | = \sup_{ | \alpha | = m } | D^\alpha \uv |. $$
The standard Sobolev space $W^{m,p}(\R^3)$ is equipped with the norm 
$$ || w ||_{m,p} = \sum_{j=0}^m || D^j w ||_{L^p(\R^3)},$$
$H^m(\R^3) = W^{m,2}(\R^3)$.
 In this section, we assume temporarily that $\uv_0 \in C^1(\R^3) \cap H^1(\R^3) \cap L^\infty(\R^3)$ and satisfies $\div \,  \uv_0 = 0$, and $A \in C(\R_+, C^1(\R))$.  

\begin{definition}\label{def:regular}  We say that $(\uv, p)$ is a regular solution of the NSE (\ref{eq:NSE}) over the time interval $ [0, T[$,   if 
\begin{enumerate}[i)]
\item $\uv, \p_t \uv, \g \uv, D^2 \uv, p, \g p$ are well defined and continuous in $t$ and $x$ for $(t, \x) \in  \, ]0, T[ \times \R^3$, and satisfy the relations (\ref{eq:NSE}.i) and (\ref{eq:NSE}.ii) in 
$\R^3$ at all $t \in \, ]0, T[$.  
\item $\forall \, \tau < T$, $\uv \in L^\infty([0, \tau], L^2(\R^3)^3) \cap L^\infty([0, \tau] \times \R^3)^3$, 
\item $(\uv(t, \cdot))_{t >0}$ uniformly converges to $\uv_0$ and in $H^1(\R^3)^3$ as $t \to 0$. 
\end{enumerate} 
\end{definition} 
The pressure at any  time $t$ is solution of the elliptic equation 
$$ \Delta p =  \div [ \div ( -  \uv \otimes \uv + A \g \uv)] ,$$
which gives $p$ once $\uv$ is calculated. 
This is why the velocity $\uv$ is sometimes referred to as the solution of the NSE instead of $(\uv, p)$, for which we set: 
\begin{eqnarray}  && \label{eq:w(t)def}  \label{eq:w(t)} W(t) = \inte |\uv (t, \x) |^2 d \x = || \uv (t, \cdot) ||_{0, 2}^2, \\
&&  \label{eq:jdet} J(t) = || \g \uv (t, \cdot) ||_{0, 2}, \\
&& \label{eq:vdet} \zoom V(t) = \sup_{\x \in \R^3} | \uv (t, \x) |  = || \uv (t, \cdot) ||_{0, \infty},   \\
&& \label{eq:v_m} \zoom V_m(t) = \sup_{\x \in \R^3} | D^m\uv (t, \x) |  = || D^m\uv (t, \cdot) ||_{0, \infty}.
\end{eqnarray}
At this stage, one of this quantity could be infinite at some date $t$.
We also set  
\BEQ  \label{eq:K}  K_A(t) = \left ( \inte A(t, \x) | \g \uv (t, \x)  |^2 d\x  \right )^{1 \over 2} = || \sqrt {A(t, \cdot)} \g \uv (t, \cdot) ||_{0, 2}. \EEQ 
\begin{definition}\label{def:singulier} 
We say that the solution becomes singular at $T$ if 
$V(t) \to \infty$ when $t \to T$, $t<T$. 
\end{definition}
We already know from J. Leray \cite{Ler1934}:
\BTHM  \label{thm:regular_A=0} Assume  $A = 0$. Then there exists $T = O (\nu V^{-2} (0) )$  such that the NSE (\ref{eq:NSE}) have a unique regular solution $(\uv, p)$ over the time interval $[0, T[$, 
 which
satisfies in addition, $\uv \in L^2([0,T], H^1(\R^3)^3) \cap C([0, T], L^2(\R^3)^3)$, 
and verifies  the energy equality for any $t \in \,  [0, T [$, 
\BEQ {1 \over 2} W(t) + \nu \int_0^t J^2 (t') dt' = {1 \over 2} W(0). \EEQ
Finally, if $\nu^{-3}  W(0) V(0)$ or $\nu^{-4} W(0) J^2(0)$ is small enough, the solution has no singularity and can be extended for all $t \in [0, \infty[$.  
\ETHM 
J. Leray also proved that the regular solution is of class $C^\infty$ is space and time in the interval $]0, T[$, the quantities $|| \uv(t, \cdot) ||_{m, 2}$ and $V_m(t)$ being controled by 
$W(0)$, $V(0)$ and $J(0)$. Unfortunately,  we are not able to generalize these results 
when $A \not= 0$ (see aditional comments in subsection \ref{sec:additional_comments}). 
\begin{Remark} Since J. Leray,  various definitions of regular solutions  to the NSE when $A=0$ and  results of local strong solutions have  been established by many different techniques. See for instance 
Fujita-Kato \cite{FK64} and Kato \cite{TK72}, as well as Meyer-Cannone \cite{CM95}, Chemin \cite{JYC11} and Chemin-Gallagher \cite{CG09} for further developements and references inside. 
\end{Remark} 
\subsection{Turbulent solutions}The notion of turbulent solution is based on a variational formulation and the energy inequality. The choice of the 
test vector fields space is essential. Within our framework, 
the space $\hbox{E}_\sigma$ specified below seems natural:
\BEQ \begin{array}{l} \zoom \hbox{E}_\sigma = \left \{ \wv \in L^1_{loc} (\R_+, H^3(\R^3)^3) \quad \hbox{s.t.}  \quad \wv \in C(\R^+, L^2(\R^3)^3),  
\phantom{\p \wv \over \p t} 
\right. \\  \hskip 2cm  \zoom \left.
 \, \g \wv \in L^\infty (\R, C_b(\R^3)^3), \quad  {\p \wv \over \p t} \in L^1_{loc} (\R_+, L^2(\R^3)^3), \quad  \div \, \wv = 0 \right \}.
 \end{array} 
\EEQ
This choice will be clear by the end of the paper. 
As usual, to find out the variational formulation, we take the dot product of the equation with a vector test field $\wv \in \hbox{E}_\sigma$ and we apply the Stokes formula,  if  the {\sl \`a priori} solution $(\uv,p)$ and its derivatives  satisfy suitable integrability conditions, what we assume at this stage. 
The time derivative $\p_t \uv$ also adresses an issue. In our formulation, it is considered in the sense of the distribution.  
 Then we formally get at a given time $t$:
\BEQ \label{eq:weak_form} \left \{  \begin{array}{l} \zoom \inte \uv_0(\x) \cdot \wv (0, \x) d\x  =   \inte \uv (t, \x) \cdot \wv (t, \x) d\x  \\~ \\ \zoom  + \int_0^t \inte [\uv (t', \x)  \otimes \uv (t', \x)] : \g \wv(t', \x) d\x dt' \\~ \\  \zoom 
-  \int_0^t \inte \uv(t', \x) \cdot \left [ \div ( (\nu + A(t', \x) ) \g \wv (t', \x) ) + {\p \wv (t', \x) \over \p t' } \right ]  d\x dt'. 
\end{array}  \right. 
\EEQ 
Notice that as $(\g p, \wv) = 0$ because $\div \,  \wv = 0 $, the pressure   is missing from this variational formulation, only involving the velocity $\uv$, which is 
standard in NSE's framework. 
\begin{definition}\label{def:turb_sols}  Let $\uv_0 \in L^2(\R^3)^3$ such that $\div \, \uv_0 = 0$. We say that $\uv = \uv (t, \x)$ is a turbulent solution to the NSE (\ref{eq:NSE}) with $\uv_0$ as initial datum,  if the following conditions are fulfilled: 
\begin{enumerate}[i)] \item For all $t \ge 0$, $\uv(t, \cdot)  \in L^2(\R^3)^3$, 
\item $\uv \in L^2(\R_+, H^1(\R^3)^3)$ and the following energy inequality holds:
\BEQ \label{eq:energy_inequality} {1 \over 2} W(t) + \nu \int_0^t J^2(t') dt' + \int_0^t K_A(t') dt'  \le {1 \over 2} W(0), \EEQ 
\item For all $t \ge 0$, and for all $\wv \in \hbox{E}_\sigma$, equality (\ref{eq:weak_form}) holds. 
\end{enumerate} 
\end{definition}
It is easily checked that when $\uv$ satisfies the items i) and ii) in Definition \ref{def:turb_sols}, then all the integrals in 
(\ref{eq:weak_form}) are well defined for whatever $\wv \in  \hbox{E}_\sigma$. 
\begin{assumption}\label{ass:donnee_initiale} To avoid repetition, we will assume throughout the rest of the paper that $\uv_0 \in L^2(\R^3)^3$ and  $\div \, \uv_0 = 0$. 
\end{assumption}
\BTHM \label{thm:Leray} (J. Leray \cite{Ler1934}) Assume $A = 0$. Then   there exists  a turbulent solution to the NSE (\ref{eq:NSE}).  
Moreover the turbulent solution becomes regular  on the interval $] C W(0)^2 / \nu^5, \infty[$, for some constant $C$. \ETHM
\begin{Remark} Leray was considering  a test vector field made of $C^\infty$ vector field in space and time, which does not change much.
\end{Remark}
By the end of the paper, we will have proved: 
\BTHM  \label{thm:turb_exist} Assume 
\begin{enumerate}[i)] 
\item $A \ge 0$ a.e in $\R_+ \times \R^3$,  
\item $A \in C_b(\R_+, W^{1, \infty}( \R^3))$, 
\item $A$ is with compact support in space uniformly in $t$, which means that there exists $R_0 >0$, such that 
$\forall \, t \ge 0$, $\forall \, \x \in \R^3$ s.t. $| \x | \ge R_0$, $A(t, \x) = 0$. 
\end{enumerate} 
Then the NSE (\ref{eq:NSE}) have a turbulent solution. 
\ETHM 
In the statement above, $C_b$ refers to as continuous bounded functions.  The proof is based on  regularizing the NSE  by means of mollifiers sized by a parameter $\E >0$, then taking the limit when $\E \to 0$. 
\begin{Remark} We do not know if the turbulent solution becomes regular after a given time $T$ when $A \not=0$ (see section \ref{sec:additional_comments} for additional comments). \end{Remark} \begin{Remark} Solutions based on a variational formulation like (\ref{eq:weak_form}) are sometimes referred  to as ``very weak solutions" (see  Lions-Masmoudi 
\cite{LM01}). 
\end{Remark}

\section{Regularized system}\label{sec:a_priori_analysis} \label{sec:regularized_system} 

\subsection{Mollifier} 

Let $\rho \in C^\infty(\R^3)$ denotes a non negative function with compact support such that
$$ \inte  \rho(\x) d\x = 1,$$
and let 
$$\rho_\E (\x) = {1 \over \E^3}\, \rho \left ( { \x \over \E} \right ).$$
Any $U \in L^1_{\hbox{\tiny loc} } (\R^3)$ being given, we set 
$$ \overline U(\x) = \rho_\E \star U (\x) = \int_{\R^3} \rho_\E \left ( \x - \yv \right ) U(\yv) d\yv. $$
It is well known that  $\overline U$ is of class $C^\infty$ and  when $U \in L^p(\R^3)$, $1 \le p < \infty$, then $\overline U$ converges to $U$ in $L^p(\R^3)$ when $\E \to 0$. We will need the following formal estimates: 
\begin{eqnarray} 
&& \label{eq:conv_ineq_0} ||  \overline U ||_{0,\infty} \le || U ||_{0,\infty}, \\
&& \label{eq:conv_ineq}  || D^m   \overline U ||_{0,\infty} \le {C_{m} \over \E^{ {3 / 2} + m} } || U ||_{0, 2}, \phantom{\int_0^1} \\
&& \label{eq:conv_ineq_2}  || \overline U - U ||_{m, 2} \le C_m \, \E \,  || U ||_{m-1, 2},  
\end{eqnarray}  
where $C_{m}  $ is a constant that only depends on $m$, the shape of $\rho$ and its derivatives. These estimates as well as many others properties about regularization by convolution can be found for instance in \cite{brezis, Ler1934, MB02}.  
\smallskip

Finally, we assume that the kernel $\rho$ is an even function, so that  the regularization operator $U \to \overline U$ is self adjoint in $L^2$. 

\subsection{Approximated system} 

We regularize the convection and the eddy viscosity terms as follows:
\begin{enumerate}[i)]
\item Following J. Leray, the convective term $\uv \cdot \g \uv$ is approximated by $\um \cdot \g \uv$,  
\item  The eddy viscosity term $-\div(A \g \uv)$ is approximated by $- \overline{ \div \, ( A \g \um)}$. 
\end{enumerate} 
This way of regularizing the eddy viscosity term provides the advantage that it preserves 
its dissipative feature. Indeed, we formally have by the self adjointness of the bar operator:
\BEQ (- \overline{ \div \, ( A \g \um) }, \uv ) = ( - \div \, ( A \g \um), \um) = (A \g \um, \g \um) = \inte A | \g \um |^2 \ge 0, \EEQ
as $A \ge 0$, and where $(\cdot, \cdot)$ denotes the scalar product in $L^2$.

According to Assumption \ref{ass:donnee_initiale}, the initial datum also needs to be regularized. Thus, we recall what is the regularized NSE, already written in the introduction: 
\BEQ \label{eq:NSE_eps} \left \{ \begin{array}{ll} 
\p_t \uv + \um \cdot \g \uv - \nu \Delta \uv -  \div \, ( \overline{A \g \um}) + \g p = 0, & \hskip 2cm\hbox{(i)} \\
\div \, \uv = 0,  \phantom{\int }   & \hskip 2cm \hbox{(ii)}\\
\uv_{t=0} = \overline{ \uv_0}. & \hskip 2cm \hbox{(iii)}
\end{array} \right. 
\EEQ
 
We adopt for the regularized NSE  (\ref{eq:NSE_eps}), the notion of regular solution given by Definition \ref{def:regular}. 
By the end of the next section, we will have proved: 
\BTHM \label{thm:exist_approch} Assume  $A \ge 0$, $A \in C_b(\R_+ , L^\infty( \R^3)^3)$. 
Then  the regularized NSE  (\ref{eq:NSE_eps}) have a unique regular solution $(\uv, p) \in C (\R_+, H^m(\R^3)^3 \times H^m(\R^3))$, $\forall \, m \in \N$,  which satisfies  the energy balance 
\BEQ \label{eq:energy_balance}  {1 \over 2} W(t) + \nu \int_0^t J^2(t') dt' + \int_0^t  K_{A,\E}^2(t') dt' = {1 \over 2} W_\E(0). \EEQ 
\ETHM
Recall that $W(t) = || \uv (t, \cdot)||_{0,2}^2$ and $J(t) = || \g \uv (t, \cdot) ||_{0,2}$ were initially defined by (\ref{eq:w(t)}) and (\ref{eq:jdet}). The quantity
$W_\E(0) = || \overline {\uv_0} ||_{0,2}^2$ verifies
\BEQ \label{eq:w(o)} W_\E(0)  \le W(0) = \inte | \uv_0 (\x) |^2 d\x. \EEQ
We also have set
\BEQ \label{eq:KAEpsilon}  K_{A,\E}(t) = \left ( \inte A(t, \x) | \g \um (t, \x) |^2 d\x \right )^{1 \over 2} . \EEQ
A similar result is in  Leray's paper \cite{Ler1934} when $A = 0$, section 26. 
His argument, based on the control of $V(t) = || \uv (t, \cdot)||_{0, \infty}$, does not work when $A \not=0$ (the main reason is detailed in section \ref{sec:additional_comments}). This is why  we had to write an original proof of Theorem \ref{thm:exist_approch} when $A \not=0$, based  this time on the control of the $H^m$ norms of the velocity, {\sl i.e.} $|| \uv (t, \cdot) ||_{m, 2}$, for any $m \ge 0$. To do so, we will find  out sharp estimates,  the control parameters of which 
are the $L^2$ norm of $\uv_0$ and $\E$. This led us to make improvements in the understanding of the equivalence between the equations and the integral representation, as well as in the processing of the 
pressure by modern regularity  results, Sobolev spaces and the Calder\`on-Zigmund Theorem \cite{ES70}, which was not known as J. Leray was writing his paper.

\begin{Remark} The assumption ``$A$ is with compact support in space uniformly in time" is not needed in this statement. Note that no further information about its gradient is required at this stage. 
\end{Remark}
\begin{assumption}\label{ass:hyp_turbu}  Throughout the rest of the paper, we will assume at least that $A \ge 0$, $A \in C_b(\R_+ , L^\infty( \R^3))$.\end{assumption} 
\subsection{Oseen representation}
The proof of Theorem \ref{thm:exist_approch} is based on an integral formulation of the regularized NSE (\ref{eq:NSE_eps}) by a suitable Kernel known as the Oseen's potential, recalled in this subsection. 
\smallskip

Let us consider the evolutionary Stokes problem with a continuous source term ${\bf f}$ and a continuous initial datum $\vv_0$: 
\BEQ \label{eq:Stokes_Problem} \left \{ \begin{array}{l} 
\p_t \uv - \nu \Delta \uv  + \g p = {\bf f}, \\
\div \, \uv = 0, \\
\uv_{t=0} = \vv_0.  
\end{array} \right. 
\EEQ
C. Oseen \cite{CO11, CWO27} shown that there exists a tensor ${\bf T}=(T_{ij} )_{1 \le i, j \le 3}$ such that if $(\uv, p)$ is a regular solution of (\ref{eq:Stokes_Problem}), then the velocity $\uv$ solution of 
(\ref{eq:Stokes_Problem})
verifies 
\BEQ \uv (t, \x) = (Q \star \vv_0) (t, \x) + \int_0^t \inte  {\bf T} (t-t', \x - \yv) \cdot {\bf f}(t', \yv) \, d \yv, \EEQ
where 
\BEQ Q(t, \x) = {1 \over (4 \pi \nu t)^{3/2} }e^{- {| \x |^2 \over 4 \nu t}} ,\EEQ
is the heat kernel and 
$$  (Q \star \vv_0) (t, \x) = \inte Q(t, \x - \yv ) \vv_0(\yv) d\yv. $$
The components of ${\bf T}$ 
can be specified as follows.  Let $$G(t,x)=\frac{1}{\abs{x}}\int_0^{\abs{x}} \frac{e^{-\frac{\rho^2}{4\nu t}}}{\sqrt{t}} \dd \rho.$$
The function $G$ satisfies the PDE
\BEQ \Delta \left ( \p_t G + \nu \Delta G \right ) = 0, \EEQ
and the Oseen's tensor is given by 
\BEQ \forall \,  i \not= j \not=k, \quad T_{ii} = - {\p^2 G \over  \p x_j^2} - {\p ^2G \over  \p x_k^2}, \quad \forall \, i \not=j, \quad T_{ij} = {\p^2 G \over  \p x_j\p x_i}.  \EEQ 
This tensor satisfies the inverse Euler system, where $L_i = (T_{i1}, T_{i2}, T_{i3})$,
\BEQ \left \{ \begin{array} {l} \zoom \p_t L_i + \nu \Delta L_i - \nabla {\p \over \p x_i} \left (\p_t G + \nu \Delta G \right ) = 0, \\
\div \, L_i = 0. \end{array} \right. 
  \EEQ
In Appendice \ref{app:ossen}, the following estimates are proved:
\begin{eqnarray}  \label{eq:estimate_oseen} | {\bf T}(t, \x ) |  \le {C \over (| \x |^2 + \nu t)^{3 \over 2}}, \\ 
\label{eq:estimate_oseen_m} \forall \, m \ge 0, \quad  | D^m  {\bf T}(t, \x ) | \le {C_m \over (| \x |^2 + \nu t)^{m+3 \over 2}}, \end{eqnarray}
$C$ and $C_m$ being some constants. We start with the following regularity result.
\BL \label{lem:integrations_oseen} For all $T >0$, $m \ge 0$, $D^m {\bf T}  \in L^p([0, T], L^q(\R^3)$, for exponents $(p, q)$ such that $q > 3/ (m+3)$ and $p < (3/2) q'$, where $1/q + 1/q' = 1$. 
\EL 
\begin{proof} By the estimate (\ref{eq:estimate_oseen_m}) we get
\BEQ  \inte | D^m {\bf T} (t, \x) |^q d\x \le 
 C \int_0^\infty { r^2  dr \over (r^2 + \nu t )^{ q(m+3) \over 2 } }  = { 1 \over (\nu t )^{q(m+3)- 3 \over 2 } } 
\int_0^\infty {\rho^2 d\rho \over (\rho^2 + 1 )^{ q(m+3) \over 2 } }.
\EEQ 
Therefore,  $D^m {\bf T}  \in L^p([0, T], L^q(\R^3)$ for $(p,q)$ such that 
$$  {q(m+3)  } - 2 >1, \quad {p (q(m+3)- 3) \over 2q } <1, $$
hence the result. 
\end{proof} 
In particular, for $m=1$, we have the following corollary: 
\begin{corollary} Let $ t > 0$. Then  $t' \to || \g \, {\bf T} (t-t', \cdot ) ||_{0, 1} \in L^1([0, t])$ and, 
\BEQ \label{eq:calcul_integral}  || \g \, {\bf T}  (t-t', \cdot ) ||_{0, 1} \le { C \over \sqrt{\nu (t-t')} } .\EEQ
\end{corollary}
Before stating the next result, we must specify some notations. Let $V = (V_{ij})_{1 \le i, j \le 3}$, $W=(W_{ij})_{1 \le i, j \le 3}$  two tensors, $V \cdot W = (V_{ij} W_{jk})_{1 \le i, k \le 3}$ their product . The vector 
$\div V$ is given component by component by:
\BEQ [\div  V]_i = \p_j V_{ij}, \EEQ
where $\p_j = \p / \p x_j$. The vector $\g V : W$ is given by
\BEQ [\g V : W]_i = \p_k V_{ij} W_{jk}. \EEQ
Finally
\BEQ \div^2 V = \p_i \p_j V_{ij}. \EEQ
\BL \label{lem:switch_integral} Let $(\uv,p)$ be a regular solution of the regularized NSE (\ref{eq:NSE_eps}) over the time interval 
$[0, T[$ for a given $T>0$. Then the velocity $\uv$ satisfies for all $t \in [0, T[$, 
\BEQ \label{eq:oseen_rep}   \left \{ \begin{array} {l} \uv (t, \x) =  (Q \, \star \,   \overline{ \uv_0}) (t, \x)  \, \\~ \\ \zoom 
 \zoom  + \int_0^t \inte \g {\bf T} (t-t', \x - \yv) :  [\um (t', \yv) \otimes \uv (t', \yv) - \overline{A\g \um } (t', \yv)  ] \, d \yv dt',  \end{array} \right. \EEQ
and the pressure $p$ is deduced from the velocity $\uv$ by the formula:
 \BEQ \label{eq:pression_reg} p (t, \x) = {1 \over 4 \pi} \div^2 \int_{\R^3} {\um (t, \yv ) \otimes \uv(t, \yv) -\overline{A  \g \um} (t, \yv) \over | \x - \yv |}    d\yv  , \EEQ
 \EL
 \begin{proof} Let 
 \BEQ \label{eq:definition_de_X} {\bf X} (t, x) = \um (t, \x) \otimes \uv (t, \x) - \overline{A\g \um } (t, \x),\EEQ
 so that the regularized NSE (\ref{eq:NSE_eps}) can be written in the form
 \BEQ \label{eq:Stokes_Problem_reg} \left \{ \begin{array}{l} 
\p_t \uv - \nu \Delta \uv  + \g p = - \div {\bf X} , \\
\div \, \uv = 0, \\
\uv_{t=0} = \overline{\uv_0}.  
\end{array} \right. 
\EEQ
 The proof is divided in two steps: 
 \smallskip
 
\hskip 1cm {\sl Step 1)}  We first study the regularity of ${\bf X}$ in order to obtain the formula (\ref{eq:pression_reg}) and the formula: $\forall \, t \in [0, t[$, 
  \BEQ \label{eq:oseen_formula_vitesse}  \uv (t, \x) =  (Q \, \star \,   \overline{ \uv_0}) (t, \x)  + \div \int_0^t \inte {\bf T} (t-t', \x- \yv) \cdot {\bf X}(t', \yv) dt'. \EEQ
\hskip 1cm {\sl Step 2)} We prove that we can switch the integral and the derivative in the formula (\ref{eq:oseen_formula_vitesse}). 
\smallskip

{\sl Step 1)}
 On one hand we have
 $$ || \um (t', \cdot) \otimes \uv (t', \cdot) ||_{0, 2} \le || \um (t, \cdot) ||_{0, \infty}  || \uv (t, \cdot) ||_{0, 2},$$
 and by (\ref{eq:conv_ineq}),    
 $$ || \um (t, \cdot) ||_{0, \infty} \le  C \E^{-{3 \over 2} } || \uv (t, \cdot) ||_{0, 2} ,$$
 which leads to
\BEQ  || \um (t', \cdot) \otimes \uv (t', \cdot) ||_{0, 2} \le C \E^{-{3 \over 2} } W(t). \EEQ
On an other hand, similar calculus inequalities lead to
\BEQ || \overline {A \g \um }(t', \cdot) ||_{0, 2} \le C || N_A ||_\infty \E^{-1} \sqrt {W(t)} , \EEQ
where we have set 
\BEQ N_A (t) = || A(t, \cdot ) ||_{0, \infty}. \EEQ
Therefore, 
\BEQ \label{eq:estimation_L2norm_X}  || {\bf X} (t, \cdot) ||_{0, 2} \le C \E^{-1} \left [  \E^{-{1 \over 2}} W(t) + || N_A ||_\infty \sqrt {W(t) } \right ] . \EEQ
Then according to  the item ii) in the definition (\ref{def:regular}), 
\BEQ \label{eq:reg_1_X}   \forall  \, \tau \in [0, T[, \quad {\bf X} \in L^\infty ([0, \tau], L^2(\R^3)^9). \EEQ
 By (\ref{eq:conv_ineq}) combined with (\ref{eq:413}), we obtain 
\BEQ \label{eq:est12}  || \overline{A(t', \cdot) \g \um (t, \cdot)} ||_{0, \infty} \le C \E^{-5/2} || N_A ||_\infty \sqrt {W(t)},\EEQ
and by (\ref{eq:conv_ineq_0}), we finally 
have
\BEQ \label{eq:estimate_prems_x_infty}  || {\bf X} (t, \cdot) ||_{0, \infty} \le V(t)^2 + C \E^{-5/2} || N_A ||_\infty \sqrt {W(t)}. \EEQ
Then, again by  the item ii) in the definition (\ref{def:regular}), for any $\tau <T$, 
\BEQ \label{eq:reg_2_X}  \forall  \, \tau \in [0, T[, \quad  {\bf X} \in L^\infty ([0, \tau] \times \R^3)^9. \EEQ
Moreover, $D^\alpha {\bf X}$ is continuous whatever $| \alpha | = 2$, in view of item i) of Definition \ref{def:regular} and the regularizing effect of the bar operator. Therefore, (\ref{eq:reg_1_X}) and (\ref{eq:reg_2_X})  being satisfied,  we can apply 
 the lemma 8 in \cite{Ler1934} and we get (\ref{eq:pression_reg})  as well as  (\ref{eq:oseen_formula_vitesse}).

\smallskip
{\sl Step 2)}\footnote{Leray states his Lemma  8 as a consequence of a uniqueness result. If the uniqueness result is entirely proved, there is in his paper 
\cite{Ler1934} no proof of thislemma 8, although it is quite reasonnable. In this present step 2), we are kniting things  backward and we indirectly  more or less prove this lemma 8, without considering the uniqueness argument, based on the $L^2$ integrabilities of ${\bf X}$ and $\uv$, which holds in our case. This proof  mainly explains the underlying machine for the 
derivation  of the following $H^m$ estimates.}  
 In what follows we set
\BEQ N_{\tau, {\bf X} } = || {\bf X} ||_{L^\infty ([0, \tau] \times \R^3)}. \EEQ
Let 
 \BEQ V(t, \x) = \int_0^t \inte {\bf T} (t-t', \x- \yv) \cdot {\bf X}(t', \yv) dt'. \EEQ
 Let ${\bf h}_i = h {\bf e}_i$, for $i = 1, 2, 3$.  Let $V_h(t, \x)$ denotes the function 
$$ V_h(t, \x) =  {1 \over h} \left [ V(t, \x+ {\bf h}_i)- V(t, \x)  \right ] . $$
on one hand we have,
 \BEQ \label{eq:switch}  \p_i V (t, \x) = \lim_{h \to 0} V_h(t, \x), \EEQ
 and on the other hand, 
 \BEQ V_h(t, \x) = 
  \int_0^t \inte {1 \over h} [ {\bf T} (t-t', \x- \yv+ {\bf h}_i) - {\bf T} (t-t', \x- \yv)] : {\bf X}(t', \yv) dt'.
 \EEQ
 Let $U_h(t, t'; \x, \yv)$ denotes the function 
 $$U_h(t, t'; \x, \yv) =  {1 \over h}  [ {\bf T} (t-t', \x- \yv+ {\bf h}_i) - {\bf T} (t-t', \x- \yv)] : {\bf X}(t', \yv), $$
 so that 
 $$ V_h(t, \x)  = \int_0^t \inte U_h(t, t'; \x, \yv)  d \yv dt'.$$
 We will pass to the limit in this integral for $h \to 0$, by two consecutive applications of the Lebesgue's theorem. 
 By definition, for any given $0 \le t' \le t$, $\x, \yv \in \R^3$, 
 \BEQ \lim_{h \to 0} U_h(t, t'; \x, \yv) = \p_i {\bf T} (t-t', \x- \yv) : {\bf X}(t', \yv). \EEQ
 We write the standard formula:
\BEQ \label{eq:acc_finis}  {1 \over h} [ {\bf T} (t-t', \x- \yv+ {\bf h}_i) - {\bf T} (t-t', \x- \yv)] = \int_0^1 \p_i {\bf T} (t-t', \x- \yv+ s {\bf h}_i) ds. \EEQ
 By consequence, by (\ref{eq:estimate_oseen_m}), we have for any fixed $0 \le t' < t$, $\x \in \R^3$, and any $h$ such that $| h | \le 1/2$, 
$$ | U_h(t, t'; \x, \yv) | \le C  N_{t, {\bf X} }  \left [ {{\rm 1} \hskip -0,1cm {\rm I}_{ B(\x, 1) } \over \nu^2 (t-t')^2}+  {{\rm 1} \hskip -0,1cm {\rm I}_{ B(\x, 1)^c } \over ( | \x - \yv |^2 / 4 + \nu (t-t') )^2 }  \right ] = H(t, t'; \x, \yv),  $$
and we observe that $ H(t, t'; \x, \yv) \in L^1_{\yv} (\R^3)$ for any given $(t, t'; \x)$. Then by Lebesgue's Theorem, 
\BEQ \label{eq:switch_2}  \lim_{h \to 0} \inte  U_h(t, t'; \x, \yv) d\y = \inte  \p_i {\bf T} (t-t', \x- \yv) : {\bf X}(t', \yv) d \yv = v(t, t'; \x),\EEQ
for all $0 \le t' < t$, $\x \in \R^3$. Let 
$$ v_h (t, t'; \x) = \inte  U_h(t, t'; \x, \yv) d\yv. $$
In other words, we have proved  that for any fixed $(t, \x) \in \,  ]0, T[ \times \R^3$, 
$$ \forall \, t' \in [0, t[, \quad \lim_{h \to 0} v_h (t, t'; \x)  = v(t, t'; \x).  $$
Notice that 
$$ \int_0^t  v_h (t, t'; \x) dt' = V_h (t, \x), $$
so that we must take the limit in the integral above when $h \to 0$. 
By using (\ref{eq:acc_finis}) combined with (\ref{eq:estimate_oseen_m}) once again, we obtain by Fubini's Theorem, 
\BEQ | v_h (t, t'; \x) | \le C  N_{t, {\bf X} }  \int_0^1 ds \inte { d \yv \over ( | \x - \yv+ s {\bf h}_i |^2 + \nu (t-t') )^2},    \EEQ 
which leads to, by the same calculation as that in the proof of Lemma \ref{lem:integrations_oseen},
\BEQ  | v_h (t, t'; \x) | \le { C  N_{t, {\bf X} } \over \sqrt { \nu (t-t') }}  \in L^1([0, t]). \EEQ
Then, by Lebesgue's Theorem once again,   
\BEQ \lim_{h \to 0} \int_0^t v_h (t, t'; \x) dt'= \lim_{h \to 0} V_h(t, \x) = \int_0^tv(t, t'; \x) dt',
\EEQ 
which means  by  (\ref{eq:switch}) and (\ref{eq:switch_2}),  
$$ \p_i \int_0^t \inte {\bf T} (t-t', \x- \yv) \cdot {\bf X}(t', \yv) dt' =  \int_0^t \inte \p_i{\bf T} (t-t', \x- \yv) \cdot {\bf X}(t', \yv) dt',$$
hence formula (\ref{eq:oseen_rep}) by (\ref{eq:oseen_formula_vitesse}).  
 \end{proof}
 \begin{Remark} By a similar reasonning based on Lebesgue's Theorem, we also can prove that $\uv \in C([0, T[, L^2(\R^3)^3) \cap C([0, T[, L^\infty(\R^3)^3)$. More generally, the time continuity of the velocity  in $H^m$  will be proved in Lemma \ref{lem:energy} below by directly using the equation. 
 \end{Remark}\section{A priori estimates and energy balance} \label{eq:Energy_balance} 
 
 We derive from the integral representation (\ref{eq:oseen_rep}) {\sl \`a priori} $H^m$ and $W^{m, \infty}$ estimates satisfied by the velocity part of any regular solution of the 
 regularized NSE (\ref{eq:NSE_eps_intro}). We start with local time estimates which we afterwards  extend for all time. We also focus on the obtention of energy equalities, which requires  
 estimates for the pressure. 
 
The constants involved in the inequalities of this section, depend on those involved in the inequality (\ref{eq:estimate_oseen_m}), on the shape of $\rho$ and its derivatives, which will be not 
systematically mentioned.

Throughout this section \ref{eq:Energy_balance}, $(\uv, p)$ denotes a regular solution of the regularized NSE  (\ref{eq:NSE_eps}). Moreover, Assumption 
\ref{ass:donnee_initiale}  and Assumption \ref{ass:hyp_turbu} hold. 
\subsection{Local time $L^2$ and $L^\infty$ estimates} \label{sec:loc_time_est}
We recall that the definition of a regular solution requires that the velocity satisfies $\forall \, \tau < T$, $\uv \in L^\infty([0, \tau], L^2(\R^3)^3) \cap L^\infty([0, \tau] \times \R^3)^3$. As we saw it in the previous section, 
this information  led to the integral representation formula (\ref{eq:oseen_rep}). The aim of this section is to show that   the $L^\infty(L^2)$ and $L^\infty(L^\infty)$ norms of $\uv$  are entirely driven by the initial kinetic energy, namely the quantity $W(0)$, which will be derived from this integral representation  formula (\ref{eq:oseen_rep}) combined with the V-maximum principle set out in Appendice \ref{ap:appen_volt}. 
 \BL \label{lem:local_energy} There exists $t_\E (W(0)) >0$ such that $\forall \, t \in [0, t_\E (W(0))]$, 
\begin{eqnarray} \label{eq:w(0)_first} && W(t) \le 4 W(0). \\
&&  \label{eq:V(t)_first} V(t) \le 2 C \E^{- {3 \over 2} }  \en. 
\end{eqnarray} 
Moreover the function 
 $x \to t_\E (x)$ is a non increasing function of $x$.  \EL
 
 \begin{proof}  We start by proving (\ref{eq:w(0)_first}) over a time interval $[0, t_{1, \E}(W(0))]$. The field  $\uv$ is a regular solution to the regularized NSE, therefore as we already have said, 
 it belongs to $L^\infty([0, T/2], L^2(\R^3)^3)$. Let 
 $$ W_{T/2} = \sup_{t \in [0, T/2]} W(t) < \infty. $$
 Assume  that 
\BEQ  4 W(0) < W_{T/2},\EEQ 
otherwise take $t_{1, \E}(W(0)) = T/2$. 
 Starting from this and working on the time interval $[0, T/2]$, we deduce from the integral representation (\ref{eq:oseen_rep}) combined  with the Young's inequality, 
 \BEQ \label{eq:est111} \sqrt {W(t) } \le \sqrt {W(0)} +
\int_0^t || \g  {\bf T}  (t-t', \cdot ) ||_{0, 1}|| {\bf X} (t', \cdot) ||_{0, 2} dt', 
\EEQ
where ${\bf X}$ is defined by (\ref{eq:definition_de_X}). 
Therefore, by (\ref{eq:calcul_integral}),  
 \BEQ \label{eq:bound_wt} \sqrt {W(t) } \le \sqrt {W(0)} + C \int_0^t { || {\bf X} (t', \cdot) ||_{0, 2} \over \sqrt { \nu (t-t') }} dt'.
 \EEQ
The estimate (\ref{eq:estimation_L2norm_X}) yields 
\BEQ \label{eq:volt_energy_1} \sqrt {W(t) } \le \sqrt {W(0)}  + C \E^{-1} \int_0^t { \E^{- 1 /2} W(t') + ||N_A||_\infty \sqrt {W(t')}   \over \sqrt { \nu (t-t') }} dt' .\EEQ
Let $P(f)$ be defined by 
\BEQ\left \{ \begin{array}{ll} P(f) = 0 & \hbox{if} \, f \le 0, \\
P(f) = C \E^{-1} [\E^{- 1 /2} f^2+ ||N_A||_\infty  f ] & \hbox{if} \,  0 \le f \le \sqrt  {W_{T/2}}, \\
P(f) = C \E^{-1} [\E^{- 1 /2} W_{T/2}+ ||N_A||_\infty  \sqrt {W_{T/2}} ]  & \hbox{if} f \le \sqrt {W_{T/2}} . 
\end{array}  \right. 
\EEQ
As $t \in [0, T/2]$ and $W(t) \le W_{T/2}$, the inequality (\ref{eq:volt_energy_1}) shows that $t \to \sqrt {W(t)} $ is a subsolution
of  the non linear Volterra equation (see Appendix \ref{ap:appen_volt})
\BEQ \label{eq:volt_energy} f(t) =  \sqrt {W(0)}  +  \int_0^t { P(f)(t')  \over \sqrt { \nu (t-t') }} dt' , \EEQ
 with the kernel 
$$  K(t) = {1 \over \sqrt {\nu t}}  \in L^1([0, T]).
$$
In this equation, $P$ is indeed a non decreasing Lipchitz continuous function. 
As $4 W(0) < W_{T/2}$, the constant function $g(t) = 2 \sqrt {W(0)}$ is a supersolution of Equation (\ref{eq:volt_energy}) over the time interval $[0, t_\E (W(0))]$, where  
\BEQ \label{eq:temps1}  t_{1,\E} (x) =  \inf \left ({ \nu \E^2 \over 4 C^2 (||N_A||_\infty  + \E^{-1/2} \sqrt {x} )^2}, {T \over 2}  \right ).  \EEQ
We then deduce from the V-maximum principle proved in Lemma \ref{lem:max_princ},  that
\BEQ \label{eq:413}  \forall \, t \in [0, t_{1,\E} (W(0))], \quad  \sqrt {W(t) } \le 2 \sqrt {W(0)}. \EEQ
Notice that the function $x \to t_{1,\E} (x)$ given by (\ref{eq:temps1}) is non increasing.
\smallskip

Let us now prove (\ref{eq:V(t)_first}). Take $t, t' \in [0, t_{1, \E}(W(0)) ]$. Combining (\ref{eq:est12}) with (\ref{eq:413}), we obtain 
\BEQ \label{eq:est1}  || \overline{A(t', \cdot) \g \um (t', \cdot)} ||_{0, \infty} \le  C \E^{-5/2} || N_A ||_\infty \sqrt {W(0)}. \EEQ
Moreover, repeating the combination of  (\ref{eq:conv_ineq}) with (\ref{eq:w(0)_first}) gives
$$ || \um (t, \cdot) ||_{0, \infty} \le  C \E^{-{3 \over 2} } || \uv (t, \cdot) ||_{0, 2} \le  C \E^{-{3 \over 2} } \sqrt {W(0)}, $$
hence
\BEQ \label{eq:est2} || \um (t', \cdot) \otimes \uv (t', \cdot)  ||_{0, \infty} \le C \E^{- { 3 \over 2}}  \sqrt {W(0)} V(t'), \EEQ
which improves the first estimate (\ref{eq:estimate_prems_x_infty}) of $ || {\bf X}(t', \cdot) ||_{0, \infty}$  by giving 
\BEQ  \label{eq:reestimateXLinfty} || {\bf X}(t', \cdot) ||_{0, \infty} \le C\E^{-{3 \over 2}} \sqrt{W(0)}  \left [ V(t') + \E^{-1} || N_A ||_\infty \right ]. \EEQ
Finally, as 
\BEQ V_\E (0) = || \overline{\uv_0} ||_{0, \infty} \le C \E^{-{3 \over 2} } || \uv ||_{0, 2} = C \E^{-{3 \over 2} } \sqrt {W(0)}, \EEQ
we get
\BEQ \label{eq:est3}  ||  Q \, \star \,   \overline{ \uv_0} (t, \x) ||_{0, \infty} \le V_\E  (0) \le C \E^{-{3 \over 2} } \sqrt {W(0)}.  \EEQ
We combine (\ref{eq:reestimateXLinfty}) and  (\ref{eq:est3}) with the formula (\ref{eq:oseen_rep}) and  (\ref{eq:calcul_integral}), which yields for any 
$t \in [0, t_{1, \E}]$, 
\BEQ \label{eq:constant_C} V (t) \le C\E^{-{3 \over 2}} \sqrt{W(0)} \left  (1+  \int_0^t {V(t') + \E^{-1} || N_A ||_\infty \over \sqrt { \nu (t-t') }} dt' \right  ).  \EEQ
Therefore, using the V-maximum principle again gives (we skip the details)
$$ \forall \, t \in [0, t_\E(W(0))], \quad  V(t) \le 2 C \E^{- {3 \over 2} } \en,$$
where, after a straighforward calculation, 
$$ t_\E(x) = \inf \left ({ 2 \E \nu  \over  ( \E^{-{1 \over 2}} ||N_A||_\infty  + 4C \sqrt {x} )^2} , t_{1, \E}  \right ) .$$
The function $x \to t_\E(x) $ is indeed non increasing. 
\end{proof} 
\subsection{Local time $H^m$ and $W^{1, \infty}$ estimates} 
To get $H^m$ estimates, uniform in time, we will argue by induction in  taking consecutive derivatives of the integral formula (\ref{eq:oseen_rep}), which comes back to the issue of switching 
integrals and derivatives. The first lemma of this section is the basis to justify the first switching, then initializing the induction. 
\BL   \label{lem:v1det} Recall that $V_1 (t) = || D \uv (t, \cdot ) ||_{0, \infty}$. Then, $\forall \, t \in [0, t_\E ( W(0))]$, 
\BEQ  \label{eq:v1det} V_1(t) \le C\nu^{- { 3 \over 2}} \E^{- { 5 \over 3}} \en \left [ \E^{- { 1 \over 3}}\en + 1 \right ] \sqrt t.\EEQ 
\EL 
\begin{proof} We can apply to this case the inequality (2.14) page 213 in Leray \cite{Ler1934}, which yields: $\forall \, t \in  [0, t_\E(W(0))]$, $\forall \, \x, {\bf y} \in \R^3$,
\BEQ | \uv (t, \x) - \uv (t, \yv) | \le C  | \x -\yv |^{1 \over 2}   \int_0^t { || ( \um \otimes \uv ) (t', \cdot)||_{0, \infty} + || \overline{A \g \um } (t', \cdot) ||_{0, \infty} \over [ \nu (t-t') ]^{3 \over 4} } dt'.
\EEQ
Hence, by (\ref{eq:V(t)_first}), (\ref{eq:est1}) and  (\ref{eq:est2}),
\BEQ  | \uv (t, \x) - \uv (t, \yv) | \le C \nu^{-{3 \over 4} } \E^{- { 5 \over 3}} \en \left [ \E^{- { 1 \over 3}}\en + 1 \right ]  | \x -\yv |^{1 \over 2} t^{1 \over 4} . \EEQ
Therefore, the inequality (2.16) page 214 in Leray \cite{Ler1934} leads to 
\BEQ V_1(t) \le C\nu^{- { 3 \over 2}} \E^{- { 5 \over 3}} \en \left [ \E^{- { 1 \over 3}}\en + 1 \right ]  \int_0^t { (t')^{1 \over 4}  \over (t-t')^{3 \over 4}} dt',
\EEQ
which gives (\ref{eq:v1det}).  
\end{proof} 
This result allows us to prove:
\BL \label{lem:local_disspipation} There exists $t^{(1)}_\E (W(0))$ such that $\forall \, t \in [0, t^{(1)}_\E (W(0))]$, 
\BEQ  \label{eq:estimate_J_V1} J(t) \le 2 C \E^{-1} \sqrt{W(0)}, \quad V_1(t) \le 2 C \E^{-{5 \over 2} } \sqrt{W(0)}. \EEQ
Moreover, the function $x \to t^{(1)}_\E (x)$ is non increasing. 
\EL 
\begin{proof} 
Let $t \le t_\E (W(0))$. 
On one hand we have 
\BEQ \forall \, \alpha = (\alpha_1, \alpha_2, \alpha_3) \in \N^3, \quad D^\alpha (Q \star \overline{\uv_0} ) = Q \star D^\alpha  \overline{\uv_0}. \EEQ
On a second hand,  we deduce from (\ref{eq:conv_ineq}), (\ref{eq:w(0)_first}), (\ref{eq:V(t)_first})  and the estimate (\ref{eq:v1det}) in Lemma \ref{lem:v1det} above, that 
$$  \g {\bf X} \in L^\infty ( [0, t^{(1)}_\E (W(0))] \times \R^3)^{27} \cap L^\infty([0, t^{(1)}_\E (W(0))], L^2(\R^3)^{27}). $$
This information combined with the estimates (\ref{eq:estimate_oseen_m}) and (\ref{eq:calcul_integral}), leads to: $ \forall \, t \in [0, t^{(1)}_\E (W(0))]$,  
\BEQ\label{eq:gradient_oseen} \g \uv(t, \x)  = Q \star \g \overline{\uv_0} + \int_0^t \inte \g {\bf T} (t-t', \x - \yv) : \g \, {\bf X} (t', \yv) \, d \yv dt'.\EEQ
The  reasoning is very close to that used in step ii) of Lemma \ref{lem:switch_integral}'s proof so that it is not necessary to repeat it. 
From (\ref{eq:gradient_oseen}), by combining Young's inequality and (\ref{eq:calcul_integral}) once again, we have
 \BEQ \label{eq:bound_wt} J(t) \le || \g \overline {\uv_0}||_{0, 2}  + C \int_0^t { dt' \over \sqrt { \nu (t-t') }} 
 \left ( || \g (\um (t, \cdot) \otimes \uv(t, \cdot) ) ||_{0, 2} + || \g \overline {A \g \um } (t, \cdot) ||_{0, 2}  \right ).  
 \EEQ
It remains to evaluate each term in the r.h.s of (\ref{eq:bound_wt}) one after another, in terms of $J(t)$, without calling on (\ref{eq:v1det}), that have served only to get (\ref{eq:gradient_oseen}). 
To begin with, notice that 
 \BEQ \label{eq:DI_grad}  || \g \overline {\uv_0}||_{0, 2}  \le C \E^{-1} \sqrt {W(0)}, \EEQ
Furthermore,  
$$\begin{array}{lll} || \g (\um (t, \cdot) \otimes \uv(t, \cdot) ) ||_{0, 2}&  \le  & || \g \um (t, \cdot)  ||_{0, \infty} || \uv(t, \cdot)  ||_{0, 2} + || \um (t, \cdot) ||_{0, \infty} J(t) \\~ \\
& \le & \zoom C \E^{-3/2} \sqrt {W(t)} J(t) \le C \E^{-3/2} \sqrt {W(0)} J(t),
\end{array} 
$$
by using  $W(t) \le 4 W(0)$ since  $t \le t_\E (W(0))$. 
Similarly, for the same reason,  we also have 
$$ || \g \overline {A \g \um } (t, \cdot) ||_{0, 2} \le C \E^{-2} || N_A ||_\infty \sqrt {W(0)}. $$
Therefore we get 
\BEQ  \label{eq:j(t)bound} J(t) \le C \E^{-1} \sqrt {W(0)} \left [ 1 +   \int_0^t { \E^{-1}  || N_A ||_\infty + \E^{-1/2} J(t') \over \sqrt { \nu (t-t') }}  dt' \right ] . \EEQ
Using the V-maximum principle once again, we deduce from (\ref{eq:j(t)bound}) that 
\BEQ \forall \, t \in [0, t^{(1)}_\E (W(0))], \quad J(t) \le 2 C \E^{-1} \sqrt{W(0)}, \EEQ
where 
\BEQ \label{eq:temps2}  t^{(1)}_\E (x) = \inf \left ( { \nu \E^2 \over \left [ || N_A ||_\infty + 2 C \E^{-1/2} \sqrt {x } \right  ]  ^2 },  t_\E (x)  \right ). \EEQ
We  note that $x \to t^{(1)}_\E (x)$ is non increasing. 
Once $J(t)$ is under control, we control  $V_1(t)$ as  in the proof of Lemma \ref{lem:local_energy}.  The estimate 
(\ref{eq:estimate_J_V1}) is now uniform in time and substancially improves (\ref{eq:v1det}). 
\end{proof} 
This result yields by induction:
\BL \label{lem:regularity_m} Let $m \ge 1$. There exists $t^{(m)}_\E (W(0))$ such that $ \forall \, t \in [0, t^{(m)}_\E (W(0))]$,
\BEQ \label{eq:est_Hm} || \uv(t, \cdot) ||_{m,2} \le C_m \E^{-m} \sqrt{W(0)}, \quad || \uv(t, \cdot) ||_{m, \infty} \le C_m \E^{-{(m+ {3 \over 2})}} \sqrt{W(0)}. \EEQ
Moreover, the function $x \to  t^{(m)}_\E (x)$ is non increasing. 
\EL
\begin{proof} Let $m \ge 1$. The property is already proved in the case $m=1$ in  lemma  \ref{lem:local_disspipation}. Let $m \ge 2$. For the simplicity, we denote by $C$ 
the constants instead of $C_k$ ($k= 1, \cdots, m$). 

Assume by induction that for any $1 \le k \le m-1$ there exists $t^{(k)}_\E (W(0)) >0$ such that 
$$0 < t^{(m-1)}_\E (W(0))  \le \cdots \le t^{(k)}_\E (W(0)) \cdots \le  t_\E (W(0)), $$ and $\forall \, t \in [0, t^{(k)}_\E (W(0))]$, 
\BEQ \begin{array}{l}  || \uv(t, \cdot) ||_{k,2} \le C\E^{-k} \sqrt{W(0)}, \\|| \uv(t, \cdot) ||_{k, \infty} \le C \E^{-{\left (k+ {3 \over 2} \right )}} \sqrt{W(0)}, \end{array}\EEQ
and $ \forall \, 1 \le k \le m-1$, the function $x \to t^{(k)}_\E (x)$ is non increasing. We first derive a bound for $|| \uv (t, \cdot) ||_{m, 2}$. Before all, we note that:
\BEQ \label{eq:initial_m}   || D^m \overline {\uv_0} ||_{0, 2} \le C  \E^{-m} \en.\EEQ
Let $t \in [0, t^{(m-1)}_\E(\en)]$. The same arguments as those developed in the proof of Lemma \ref{lem:v1det} yield by induction 
to 
$$ V_m(t) \le \varphi(t),$$
where $t \to \varphi (t)$ is a continuous function. Let 
  $\alpha = (\alpha_1, \alpha_2, \alpha_3)$, $| \alpha | = m$. From this, we get 
  $$  D^\alpha {\bf X} \in L^\infty ( [0, t^{(1)}_\E (W(0))] \times \R^3)^9 \cap L^\infty([0, t^{(1)}_\E (W(0))], L^2(\R^3)^9), $$
which gives by arguments already  set out, 
  $$ D^\alpha \uv(t, \x)  = Q \star \g \overline{\uv_0} + \int_0^t \inte \g {\bf T} (t-t', \x - \yv) : D^\alpha {\bf X} (t', \yv) \, d \yv dt'.$$
  It remains to derive from the induction hypothesis a sharp estimate of $|| D^\alpha {\bf X} (t', \yv) ||_{0, 2}$ in order to use the V-maximum principle to control 
  $|| D^\alpha \uv (t', \yv) ||_{0, 2}$. 
According to Lemma \ref{lem:local_energy}, the choice of $t$ and usual results about convolution, we have 
\BEQ \label{eq:noyau_m_estimate}  || D^\alpha \overline{A \g \um } (t, \cdot ) ||_{0, 2} \le C || N_A ||_{\infty}  \E^{-(m+1)} \sqrt{W(t)} \le 2 C || N_A ||_{\infty}  \E^{-(m+1)} \en. \EEQ
Furthermore, the Leibnitz formula gives, 
\BEQ  D^\alpha (\um \otimes \uv) = \um \otimes D^\alpha \uv +  \sum_{\substack{ \beta = (p,q,r) \\ | \beta | < | \alpha | }} 
C_{\alpha_1}^p C_{\alpha_2}^q C_{\alpha_3}^r D^\beta \um \otimes D^{\alpha- \beta} \uv.  \EEQ
We deduce  from (\ref{eq:conv_ineq}) and Lemma \ref{lem:local_energy},
$$ || D^\alpha (\um \otimes \uv) ||_{0, 2} \le C \E^{-3/2} \en \left [  || D^\alpha \uv ||_{0,2} + \sum_{| \beta | < | \alpha |} \E^{- | \beta | }  || D^{\alpha- \beta} \uv ||_{0, 2} \right ] . $$
The induction hypothesis yields 
\BEQ \label{eq:nonlinear_m}   || D^\alpha (\um \otimes \uv) ||_{0, 2} \le C \E^{-3/2} \en \left [  || D^\alpha \uv ||_{0,2} +  m \E^{-2m} \en \right ]. \EEQ
Notice that we have not optimized things above, and this last estimate could be substancially improved. However, this is not essential. 
Now, by combining (\ref{eq:estimate_oseen_m}), (\ref{eq:calcul_integral}), (\ref{eq:initial_m}), (\ref{eq:noyau_m_estimate}) and (\ref{eq:nonlinear_m}) we obtain
\BEQ  \label{eq:Da_bound}\begin{array}{l}  || D^\alpha \uv (t, \cdot) ||_{0, 2}   \le  
C \sqrt {W(0)}  \left [ \E^{-m} +  \right. \\ ~ \\ \zoom \hskip 1,5cm
 \left. \int_0^t {\E^{-3/2}  || D^\alpha \uv (t', \cdot) ||_{0, 2}  + \E^{-(m+1)} || N_A ||_\infty + 
m \E^{-(2m + {3 \over 2})} \en \over \sqrt { \nu (t-t') }}  dt'  \right ]  . \end{array} \EEQ
By the V-maximum principle, we deduce from (\ref{eq:Da_bound}) that 
\BEQ \forall \, t \in [0, t^{(m)}_\E (W(0))], \quad || D^\alpha \uv (t, \cdot) ||_{0, 2} \le 2 C \E^{-m} \sqrt{W(0)}, \EEQ
where 
\BEQ \label{eq:tempsm}  t^{(m)}_\E (x) = \inf \left ( { \nu \E^2 \over \left [ || N_A ||_\infty +  C \E^{-1/2}(1+m \E^{-m}) \sqrt {x } \right ]  ^2 },  t^{(m-1)}_\E (x) \right ), \EEQ
and $x \to t^{(m)}_\E (x) $ is clearly non increasing. By a similar process, we also found the  bound for $ V_m(t)$, therefore  $|| \uv (t, \cdot) ||_{m, \infty}$. 
\end{proof} 
\begin{Remark} We  observe that $t^{(m)}_\E(x) \to 0$ as $\E \to 0$, for fixed $x$ and $m$. Moreover, the only estimate which does not blow up when $\E \to 0$, is the estimate (\ref{eq:V(t)_first}) 
about $W(t)$. This is in coherence with all former known results about the Navier-Stokes equations. 
\end{Remark} 
In the following, we set
$$I_{\E, m} = [0, t^{(m)}_\E (W(0))] . $$
\subsection{Energy balance and transition from local to global time}  
The transition from local to global time is based on the energy balance, which remains to be justifed. To do so, we must find additional estimates 
about the pressure to check that it satisfied right integrability conditions. The pressure satisfies the elliptic equation: 
\BEQ \label{eq:eq_poisson_presion} \Delta p = \div[ \div (- \um \otimes \uv + \overline{ A \g \um } )]. 
\EEQ
Therefore, we can write:
\BEQ p (t, \x) = {1 \over 4 \pi} \inte \g^2 \left ( 1 \over r \right ) [- \um \otimes \uv + \overline{ A \g \um }](t, \yv) d \yv, \EEQ
where $\g^2 = (\p_i \p_j )_{1 \le i,j \le 3}$, $r = | \x-\yv|$. 
This expression must be understood as a singular integral operator, with a $\delta$-function for $i=j$. 
By the Calder\`on-Zigmund Theorem (see in E. Stein  \cite{ES70} and also in P. Galdi \cite{GG00}, chapters 2 and 3) and Lemma \ref{lem:local_energy}, we see that $p \in C([0, t_\E( W(0)) ], L^2(\R^3))$ and the 
$L^2$ bound is uniform in $t$. From this,  
the differential quotient method due to L. Nirenberg \cite{LN59} and the standard elliptic theory (see also in Brezis \cite{brezis}, section IX.6), combined with 
Lemma \ref{lem:regularity_m} allows to write by induction that 
$\forall \, m \ge 0$, $\forall \, t \in I_{\E, m+2}$, 
\BEQ || p (t, \cdot) ||_{m, 2} \le C(W(0), || N_A ||_\infty, \E, m). \EEQ
Having said that, we can be more specific:
\BL \label{lem:pression} Let $m\ge 4$, $t \in  I_{\E, m+4}$. Then 
\BEQ \label{eq:esti_pressure}   || p (t, \cdot) ||_{m, 2} \le C_m \E^{-m} \en \left [ \E^{-1} || N_A ||_\infty + \E^{-m } \en \right ], \EEQ
where we recall that $N_A (t) = || A(t, \cdot ) ||_{0, \infty}$. \EL
\begin{proof}  Let $t \in I_{\E, m+4}$, $\alpha = (\alpha_1, \alpha_2, \alpha_3)$ such that $| \alpha | \le m$. As $H^4(\R^3) \hookrightarrow C^2(\R^3)$,   then
$D^\alpha \uv (t, \cdot)$, $D^\alpha p(t, \cdot) \in C^2(\R^3)$, which allows to derive the regularized Navier-Stokes equations up to order $\alpha$. In particular, we have  at time $t$: 
\BEQ \label{eq:NSE_derived} \p_t D^\alpha \uv - \nu \Delta D^\alpha \uv + \g D^\alpha p = -D^\alpha (\um \otimes \uv ) + D^\alpha (\overline{A \g \um} ). \EEQ
By consequences, from $\div D^\alpha \uv = 0$ we get, 
 \BEQ \label{eq:pression_Dalpha} D^\alpha p (t, \x) = 
 {1 \over 4 \pi}   \int_{\R^3} \g^2 \left ( {1 \over r} \right ) [ D^\alpha (\um \otimes \uv ) (t, \yv) - D^\alpha (\overline{A \g \um} ) (t, \yv)] d\yv .\EEQ
 As $m \ge 4$, $H^m(\R^3)$ is an algebra, hence by (\ref{eq:est_Hm})
\BEQ || (\um \otimes \uv )(t, \cdot) ||_{m, 2} \le C || \um(t, \cdot) ||_{m,2} || \uv(t, \cdot) ||_{m, 2} \le C_m \E^{-2m} W(0). \EEQ
Therefore, combining this inequality with (\ref{eq:noyau_m_estimate}) and Calder\`on-Zigmung Theorem, we obtain
\BEQ \label{eq:non_linear_est}  \begin{array}{ll}  || D^\alpha p (t, \cdot)  ||_{0, 2}  &  \le C(|| D^\alpha (\um \otimes \uv )(t, \cdot) ||_{0, 2} +  || D^\alpha \overline{A  \g \um}(t, \cdot) ||_{0, 2} ) \\~ \\
&\zoom \le  C_m \E^{-m} \en \left [\E^{-m } \en +  \E^{-1} || N_A ||_\infty  \right ]  =P_{m, \E}. \phantom{\int_0^{e}} \end{array} \EEQ
Therefore, $p (t, \cdot) \in H^m(\R^3)$ for all $t \in I_{\E, m+4}$ and 
\BEQ \label{eq:estimer_la_pression}  || p (t, \cdot) ||_{m, 2} \le m P_{m, \E}, \EEQ
giving (\ref{eq:esti_pressure}). 
\end{proof} 
\BL \label{lem:energy} Let $m \ge 4$. Then $\uv \in C(I_{\E, m+4}, H^m(\R^3)^3)$ and  for all $\alpha$ such that $| \alpha | \le m$, the following energy balance holds: 
\BEQ \label{eq:balance_D_alpha} {1 \over 2} \inte | D^\alpha \uv (t, \x) |^2 d\x = {1 \over 2} \inte | D^\alpha \overline {\uv_0} (\x) |^2 d\x + \int_0^t \left [ I_1(t')  + I_2(t')  + I_3(t') \right ] dt',  
\EEQ
where 
$$ \left \{ \begin{array}{l} \zoom  I_1 (t') = - \inte D^\alpha (\um \otimes \uv )(t', \x) : \g D^\alpha (t', \x) d\x, \\~ \\
\zoom I_2 (t') = -\nu \inte | \g D^\alpha u (t', \x) |^2 d\x, \\~ \\
\zoom I_3 (t') = - \inte D^\alpha \overline{ A \g \um} (t', \x) \cdot \g D^\alpha \uv (t', \x) d\x.
\end{array} \right. 
$$
In particular, we have $\forall \, t \in  I_{\E, m+4}$, 
 \BEQ \label{eq:energy_balanc_2}  {1 \over 2} W(t) + \nu \int_0^t J^2(t') dt' + \int_0^t  K_{A,\E}^2(t') dt' = {1 \over 2} W_\E(0), \EEQ 
 for all $t \in [0, T[$, where $W(t)$ and $J(t)$ were initially defined by (\ref{eq:w(t)def}) and (\ref{eq:jdet}), and $K_{A, \E}$ by (\ref{eq:KAEpsilon}).
 \EL
\begin{proof} We deduce from the equation (\ref{eq:NSE_derived}),  
\BEQ \label{eq:NSE_eps_Dalpha} 
D^\alpha (\p_t  \uv) =- \div (D^\alpha (\um \otimes \uv ))  + \nu \Delta D^\alpha \uv +   \div \,  D^\alpha \overline{ A \g \um} - \g D^\alpha p.  
\EEQ
Therefore, by (\ref{eq:est_Hm}) and 
(\ref{eq:esti_pressure}), we get 
\BEQ \begin{array} {l}  ||  D^\alpha (\p_t  \uv) (t, \cdot) ||_{0, 2} \le \\ 
\hskip 2,5cm C_{m+1}\E^{-(m+1)} \en \left [  \en  (1+\E^{-m}) + \E^{-1} (\nu + || N_A ||_\infty) \right ]. \end{array} \EEQ
By consequence, $\uv, \p_t \uv \in L^\infty(I_{\E, m+4}, H^m(\R^3)^3) \subset L^2(I_{\E, m+4}, H^m(\R^3)^3)$. Therefore,  by a well known result of functional analysis 
(see in Temam \cite{RT01} for example), $\uv \in C(I_{\E, m+4}, H^m(\R^3)^3)$ and 
$$ (\p_t \uv(t, \cdot) , \uv(t, \cdot))_m =  {d \over 2 dt} || \uv (t, \cdot) ||_{m, 2}^2.$$
Following the usual process, we form the dot product of the equation (\ref{eq:NSE_eps_Dalpha}) with $D^\alpha \uv$. We integrate over $\R^3$ by using the Stokes formula, which is 
possible by the integrabilty properties of $D^\alpha \uv$ and $D^\alpha p$. In particular, since $\div D^\alpha \uv = 0$, we get 
$$ (\g D^\alpha p, D^\alpha \uv) = 0.$$
We get the energy balance (\ref{eq:balance_D_alpha}) after integrating in time over $[ 0, t]$. 
In the special case $\alpha = 0$, we obtain $(\ref{eq:energy_balanc_2})$ by noting  that in addition:
$  ( \um \cdot \g \uv, \uv) = 0 $. 
\end{proof} 
\BTHM \label{thm:regularity_result} Let $T >0$. Then\footnote{we take $m \ge 4$ to be in coherence with the arguments set out in Lemma \ref{lem:energy}. However, what can do more can do less, and the result holds for $m = 0, 1, 2, 3$.} $\forall \, m \ge 4$, $(\uv, p) \in C([0,T], H^m(\R^3) \times H^m(\R^3))$, the energy balance (\ref{eq:energy_balanc_2}) holds for all $t \in [0,T]$ as well as 
the estimates (\ref{eq:est_Hm}) and (\ref{eq:esti_pressure}). 
\ETHM
\begin{proof}  The energy balance (\ref{eq:energy_balanc_2}) combined with the inequality (\ref{eq:w(o)}),  shows that  for all $t \in I_{\E, m+4}$, we have $W(t) \le W(0)$, improving substancially the estimate (\ref{eq:w(0)_first}). 
In particular we can write
\BEQ W( \delta t_\E^{(m)}(W(0)) \le W(0), \quad  \hbox{where} \quad  \delta t_\E^{(m)}(x) = t^{(m+4)}_\E(x).
\EEQ 
Let $(t_{n, \E}^{(m)})_{n \ge 1}$ be the sequence given by 
\BEQ \label{eq:rec} t_{1, \E} ^{(m)} = \delta t_\E^{(m)}(W(0)) , \quad t_{n, \E}^{(m)} =  t_{n-1, \E}^{(m)} + \delta t_\E^{(m)} (W(t_{n-1, \E} )). \EEQ
Assume that for a given $n$, $t_{n, \E}^{(m)}$ is constructed such that $u \in C( [0, t_{n, \E}^{(m)}], H^m(\R^3))$, which holds when $n = 1$ by Lemma \ref{lem:regularity_m} since $t^{(m+4)}_\E \le t^{(m)}_\E$. In particular the energy balance holds over 
$[0, t_{n, \E}^{(m)}]$, and $W(t_{n, \E}^{(m)}) \le W(0)$, which yields $\delta t_\E^{(m)} (W(t_{n, \E}^{(m)})) \ge \delta t_\E^{(m)} (W(0))$ by the decrease of the function
$x \to  t^{(m+4)}_\E(x)$. Therefore, we can reproduce 
the arguments of the section \ref{sec:loc_time_est} and  the lemma \ref{lem:pression}  and \ref{lem:energy} from the time $t_{n, \E}^{(m)}$, by the continuity in time, thus validating the iteration $n+1$ of (\ref{eq:rec}) by the inductive hypothesis. In particular we have
$$ t_{n+1, \E}^{(m)}= \sum_{k=1} ^n \delta t_\E^{(m)} (W(t_{k, \E} ^{(m)}))\ge n \delta t_\E^{(m)}(W(0)) , $$
which is larger than $T$ for $n$ large enough, concluding the proof.  \end{proof} 
\begin{Remark} The estimates obtained in this section prove that for any $m \ge 0$, 
\BEQ || \uv(t, \cdot) - \overline {\uv_0}  ||_{m, 2} \le C_m \E^{-(2m + {3 \over 2}) } \sqrt { t \over \nu}. \EEQ
\end{Remark} 
\begin{Remark} The continuity in time ensures  that any regular  solution to the regularized NSE (\ref{eq:NSE_eps}) satisfies the following semi-group like property: 
   $\forall \, t_0 \in ]0, T[$, $\forall \, t \in [t_0, T[$, 
 \BEQ \label{eq:oseen_rep_2}   \left \{ \begin{array} {l} \uv (t, \x) =  (Q \, \star \,    \uv(t_0, \cdot)) (t, \x)  \, \\~ \\ \zoom 
 \zoom  + \int_{t_0}^t \inte \g {\bf T} (t-t', \x - \yv) : [\um (t', \yv) \otimes \uv (t', \yv) - \overline{A\g \um }(t', \yv)  ] \, d \yv dt'.  \end{array} \right. \EEQ
 This what is implicitely used in the proof of Theorem \ref{thm:regularity_result}. 
\end{Remark}
Once the regularity result of Theorem \ref{thm:regularity_result} is established, following the standard routine yields the  uniqueness result:
\BTHM \label{thm:uniqueness} The regularized NSE  (\ref{eq:NSE_eps}) have at the most one regular soution $(\uv, p)$.  
\ETHM

\section{Existence of solution for the regularized  NSE}
The aim of this section is the proof of the existence result stated in Theorem \ref{thm:exist_approch}. The solution of the regularized NSE (\ref{eq:NSE_eps}) is constructed by a standard Picard iteration process based on the Oseen integral representation. 
\subsection{Iterations} 
Let us put 
\BEQ \uv_1 (t, \x) =  (Q \, \star \,   \overline{ \uv_0}) (t, \x),  \EEQ
and for all $n >0$, 
\BEQ \label{eq:picard_process}   \left \{ \begin{array} {l} \uv_{n+1} (t, \x) =  (Q \, \star \,   \overline{ \uv_0}) (t, \x)  \, \\~ \\ \zoom 
 \zoom  + \int_0^t \inte \g {\bf T} (t-t', \x - \yv) : [(\um_n \otimes \uv_n)(t', \yv) - \overline{A\g \um_n } (t', \yv) ] \, d \yv dt'.  \end{array} \right. \EEQ
The first result of this section aims to check that the sequence $(\uv_n)_{n \in \N}$ makes sense and to get estimates about $\uv_n$, similar to the estimates 
(\ref{eq:est_Hm}). 
\BL  for all $m\ge 0$,  
 $n \ge 1$,   $\uv_n \in C(I_{\E, m}, H^m(\R^3))^3$, where $I_{\E, m} = [0, t^{(m)}_\E(W(0))]$, $t^{(m)}_\E(x)$ is given by (\ref{eq:tempsm}). Moreover, 
 we have, $\forall \, m \ge 0$, $\forall \, t \in [0, t^{(m)}_\E(W(0))]$, $\forall \, n \in \N$:
 \BEQ \label{eq:Vmnt}  || \uv_n(t, \cdot) ||_{m,2} \le C_m \E^{-m} \sqrt{W(0)}, \quad || \uv_n(t, \cdot) ||_{m, \infty} \le C_m \E^{-{(m+ {3 \over 2})}} \sqrt{W(0)}.\EEQ 
\EL
\begin{proof}
By recycling the proof of Lemma \ref{lem:local_energy}, we get the inequality 
\BEQ \label{eq:volt_energy_ite} \sqrt {W_{n+1} (t) } \le \sqrt {W(0)}  + C \E^{-1} \int_0^t { \E^{- 1 /2} W_n(t') + \sqrt {W_n(t')}   \over \sqrt { \nu (t-t') }} dt'.\EEQ 
A straightforward inductive  reasonning yields
\BEQ \forall \, t \in [0, t_\E(W(0))], \quad \forall n \in \N, \quad W_n(t) \le 4 W(t), \EEQ
where $t_\E(x)$ is specified by the formula (\ref{eq:temps1}), in which we take $T = \infty$. Following the proofs of Lemma \ref{lem:local_disspipation} and \ref{lem:regularity_m}
we obtain (\ref{eq:Vmnt}).  We skip the details. The continuity in time is obtained in the same way as in the item 2) of the proof of Lemma \ref{lem:switch_integral}, based on  $\g  {\bf T} \in L^1_{t, \x}$.  \end{proof}
\BL \label{lem:reg_un}  Let $m \ge 4$, $n \ge 0$. There exists $p_{n+1} \in C(I_{\E, m+4}, H^m(\R^3))$ such that $(\uv_{n+1}, p_{n+1})$ satisfies over 
$I_{\E, m+4} \times \R^3$ the evolutionary Stokes equation: 
\BEQ \label{eq:ite_Stokes_pb_n} \left \{  \begin{array} {l} \p_t \uv_{n+1} - \nu \Delta \uv_{n+1} + \g p_{n+1} = - \div ( \um_n  \otimes  \uv_n - \overline{A \g \uv_n}),  \\
\div  \, \uv_{n+1} = 0,   \\ 
\uv_{n+1}\vert _{t = 0} = \overline{ \uv_0}. 
\end{array} \right.  \EEQ
\EL
\begin{proof} Let ${\bf X}_n =   \um_n  \otimes  \uv_n - \overline{A \g \uv_n}$, and consider the evolutionary Stokes problem
\BEQ \label{eq:ite_Stokes_pb_transfer} \left \{  \begin{array} {l} \p_t \vv - \nu \Delta \vv + \g p_{n+1}  = - \div {\bf X}_n , \\
\div  \, \vv= 0,\\ 
\vv \vert _{t = 0} = \overline{ \uv_0}. 
\end{array} \right.  \EEQ
According to Lemma \ref{lem:reg_un} and because $m \ge 4$, we have at least $ {\bf X}_n \in L^2 (I_{\E, m+4}, H^3(\R^3))$. Then by Theorem 1.1 and Proposition 1.2 in Chapter 3 in Temam 
\cite{RT01}, we know the existence of a unique weak solution to the Stokes problem (\ref{eq:ite_Stokes_pb_transfer}) such that 
$$ \p_t \uv \in L^2 (I_{\E, m+4}, H^3(\R^3)), \quad \uv \in  L^2 (I_{\E, m+4}, H^5(\R^3)), \quad p_{n+1} \in L^2 (I_{\E, m+4}, H^4(\R^3)),$$
which is constructed by the Galerkin method. Therefore, the conditions for the application of Lemma 8 in Leray \cite{Ler1934}  are fulfiled, and 
\BEQ \label{eq:picard_process_equation}   \left \{ \begin{array} {l} \vv (t, \x) =  (Q \, \star \,   \overline{ \uv_0}) (t, \x)  \, \\~ \\ \zoom 
 \zoom  + \int_0^t \inte \g  {\bf T} (t-t', \x - \yv) : [(\um_n \otimes \uv_n)(t', \yv) - \overline{A\g \um_n } (t', \yv) ] \, d \yv dt',  \end{array} \right. \EEQ
 hence $\vv = \uv_{n+1}$ because the solution of (\ref{eq:ite_Stokes_pb_transfer}) is unique. From there, the regularity of $p_{n+1}$ is obtained by the same argument as in the proof of Lemma \ref{eq:esti_pressure}. 
\end{proof} 
\subsection{Contraction property} 
 In what follows, we set
\begin{eqnarray} && W_n (t) = \inte | \uv_n (t, \x) |^2 d\x = || \uv_n (t, \cdot) ||_{0, 2}^2, \\
&& V_n(t) = || \uv_n (t, \cdot) ||_{0, \infty}, \\ 
&& J_n (t) = || \g \uv_n (t, \cdot) ||_{0, 2}, \zoom \phantom{\int_0^1} \\
&& \label{eq:v_m_n} \zoom V_{m,n} (t) = \sup_{\x \in \R^3} | D^m\uv (t, \x) |  = || D^m\uv (t, \cdot) ||_{0, \infty}. \phantom{int_0^{e^x}} 
\end{eqnarray}
We prove in this section  that there exists  $\tau^{(m)}_\E(W(0)) >0$ such that the sequence $\suite \uv n$ 
satisfies a contraction property over $[0, \tau^{(m)}_\E(W(0))]$. \smallskip
Given any time $\tau >0$, we equip the space $C([0, T], H^m(\R^3))$ with the natural uniform norm 
\BEQ || w ||_{ \tau; m, 2} = \sup_{t \in [0, \tau]} || w(t, \cdot) ||_{m, 2}, \EEQ
making $C([0, T], H^m(\R^3))$ a Banach space. We show in this subsection the following 
\BL \label{lem:contraction} For all $m \ge 0$, there exists a time $\tau^{(m)}_\E(W(0))$ such that 
\BEQ \label{lem:contraction_Lemma} \forall \, n \in \N, \quad || \uv_{n+1} - \uv_n ||_{\tau^{(m)}_\E(W(0)); m, 2} \le {1 \over 2}  || \uv_{n} - \uv_{n-1} ||_{\tau^{(m)}_\E(W(0)); m, 2}. \EEQ
Moreover, the function $x \to \tau^{(m)}_\E(x)$ is non increasing. 
\EL
\begin{proof} Let $m \ge 0$. From now, we are working on the time interval $I_{\E, m+4} = [0, t^{(m+4)}_\E(W(0))]$. Let $n\ge1$, $\alpha = (\alpha_1, \alpha_2, \alpha_3)$ such that 
$| \alpha | \le m$, and $w_{n,\alpha}(t)$ defined by
\BEQ w_{n,\alpha}(t) =  || D^\alpha \uv_n(t, \cdot) - D^\alpha \uv_{n-1} (t, \cdot) ||_{0, 2}. \EEQ
We first evaluate $w_{n+1,0}(t)$ in terms of $w_{n,0}(t)$. To do so, 
let 
$$ \Delta_{n, 0} (t') =  || (\um_n \otimes \uv_n) (t', \cdot) - (\um_{n-1} \otimes \uv_{n-1}) (t', \cdot) ||_{0,2},$$
and
observe that at any time 
$t'$, we have
\BEQ \begin{array}{l}  \Delta_{n, 0} (t')    \le 
|| \um_n (t', \cdot) ||_{0, \infty} w_{n,0}(t') + || \uv_{n-1} (t', \cdot) ||_{0, \infty} || (\um_n - \um_{n-1})(t', \cdot) ||_{0, 2},
\end{array} 
\EEQ
which leads by (\ref{eq:Vmnt}) to 
\BEQ \label{eq:Delta1} \Delta_{n, 0} (t')  \le C \en w_{n,0}(t'). \EEQ
Similarly, we also have 
\BEQ \label{eq:Ag1} || \overline{A\g \um_n} (t', \cdot) - \overline{A\g \um_{n-1} } (t', \cdot) ||_{0, 2} \le C || N_A ||_\infty \E^{-1} w_{n,0}(t'). \EEQ
Inequalities (\ref{eq:Delta1}) and (\ref{eq:Ag1}) combined with the relation (\ref{eq:picard_process}) and arguments used many times before 
lead to 
\BEQ w_{n+1,0}(t) \le C  (  \en + \E^{-1} || N_A ||_\infty  ) \int_0^t {w_{n,0}(t') \over \sqrt{t-t'} } dt'. \EEQ
The same procedure leads to:   $\forall \, \alpha = (\alpha_1, \alpha_2, \alpha_3)$ such that $ | \alpha | \le m$,
\BEQ \label{eq:wn+1alpha} w_{n+1,\alpha}(t) \le 
C_{n, \alpha, \E} (\en+ || N_A ||_\infty) \int_0^t { || \uv_n(t', \cdot) - \uv_{n-1} (t', \cdot ) ||_{m, 2} \over \sqrt {\nu (t-t')} } dt'. 
\EEQ
To see this, we must first estimate $\Delta_{n, \alpha} (t')$, where
$$ \Delta_{n, \alpha} (t') =  || D^\alpha (\um_n \otimes \uv_n) (t', \cdot) - D^\alpha (\um_{n-1} \otimes \uv_{n-1}) (t', \cdot) ||_{0,2}.$$
By the Leibnitz formula, we have
$$\begin{array}{l} D^\alpha (\um_n \otimes \uv_n) -   D^\alpha (\um_{n-1} \otimes \uv_{n-1}) =  \\~ \\  \zoom 
\hskip 2cm \sum_{\substack{ \beta = (p,q,r) \\ | \beta | \le | \alpha | }} 
C_{\alpha_1}^p C_{\alpha_2}^q C_{\alpha_3}^r (D^\beta \um_{n+1} \otimes D^{\alpha- \beta} \uv_{n+1} -  
D^\beta \um_{n} \otimes D^{\alpha- \beta} \uv_{n}). \end{array} $$
We deduce from inequality (\ref{eq:Vmnt}) that
\BEQ \label{eq:dalphabeta}  \begin{array}{l}  || (D^\beta \um_{n+1} \otimes D^{\alpha  - \beta} \uv_{n+1}) (t', \cdot)-  
(D^\beta \um_{n} \otimes D^{\alpha- \beta} \uv_{n} ) (t', \cdot) ||_{0, 2} \le \\~ \\  \zoom 
\hskip 3cm 
\en \left [ C_{\beta} \E^{- |\beta|} w_{n, \alpha-\beta} (t') + C_{ \alpha - \beta } \E^{- (| \alpha| - |\beta |)} w_{n, \beta} (t') \right ] . \end{array} \EEQ
Furthermore, 
\BEQ \label{eq:dalphaA}  || D^\alpha \overline{A\g \um_n} (t', \cdot) - D^\alpha \overline{A\g \um_{n-1} } (t', \cdot) ||_{0, 2} \le C || N_A ||_\infty \E^{-(|\alpha|+1)}  w_{n,0}(t'). \EEQ
By noting that $\forall \, \beta$ s.t. $| \beta | \le m$,
$$ w_{n, \beta} (t') \le || \uv_n (t', \cdot) - \uv_{n-1}  (t', \cdot) ||_{m, 2}, $$
we get (\ref{eq:wn+1alpha}) by  combining (\ref{eq:dalphabeta}), (\ref{eq:dalphaA}) and (\ref{eq:picard_process}).   
\smallskip

Summing (\ref{eq:wn+1alpha}) over $\alpha$ for $0 \le | \alpha | \le m$, yields for all $t \in I_{\E, m}$,
$$  \begin{array}{l}  || \uv_{n+1} (t, \cdot) - \uv_{n} (t, \cdot ) ||_{m, 2} \le \\ 
\zoom \hskip 3cm C_{n, \alpha, \E} (\en+ || N_A ||_\infty) \int_0^t { || \uv_n(t', \cdot) - \uv_{n-1} (t', \cdot ) ||_{m, 2} \over \sqrt {\nu (t-t')} } dt', \end{array} 
$$
hence $\forall \, \tau \in I_{\E, m}$, $\forall \, t \in [0, \tau]$, 
$$  \begin{array}{l}  || \uv_{n+1} (t, \cdot) - \uv_{n} (t, \cdot ) ||_{m, 2} \le \\ 
\zoom \hskip 3cm C_{n, \alpha, \E} (\en+ || N_A ||_\infty) \sup_{t' \in [0, \tau]}  || \uv_n(t', \cdot) - \uv_{n-1} (t', \cdot ) ||_{m, 2} \sqrt{\tau \over \nu}. \end{array} 
$$
Consequently, (\ref{lem:contraction_Lemma}) holds by taking 
\BEQ \tau_{\E}^{(m)} (x) =\inf\left (  {\nu \over 2 C_{n, \alpha, \E} (\sqrt x+ || N_A ||_\infty)^2}, t_{m+4, \E}(x) \right ). 
\EEQ
which is indeed a non increasing function of $x$. \end{proof} 
\subsection{Concluding proof}
We are now capable of proving Theorem \ref{thm:exist_approch}. Let $m \ge 4$, so that $H^m (\R^3) \hookrightarrow C^2(\R^3)$.  For the simplicity we write $\tau$ instead of $\tau_\E^{(m)} (W(0))$.  Lemma \ref{lem:contraction} shows that the sequence 
$\suite \uv n$ converges to some $\uv$ in $C([0, \tau], H^m(\R^3)^3)$. We aim to prove that $\uv$ satisfies the Oseen integral relation (\ref{eq:oseen_rep}). 
Let $t \in [0, \tau]$, and consider
\BEQ \begin{array}{l} \zoom v_n(t, \x; t') = \inte \g{\bf T} (t-t', \x - \yv) : [(\um_n \otimes \uv_n)(t', \yv) - \overline{A\g \um_n } (t', \yv) ] \, d \yv, \\
\zoom v(t, \x; t') = \inte \g {\bf T} (t-t', \x - \yv) : [(\um\otimes \uv)(t', \yv) - \overline{A\g \um } (t', \yv) ] \, d \yv. \end{array}  \EEQ
The inequality (\ref{eq:estimate_oseen_m}) leads to, for $t' < t$, 
$$ \begin{array}{l} \zoom  | v_n(t, \x; t') - v(t, \x; t') | \le \\~ \\  \zoom 
\hskip 2cm {C \over [\nu (t-t')]^2 } \inte \left [  | (\um_n \otimes \uv_n- \um\otimes \uv)(t', \yv)|   + | ( \overline{A\g \um_n } -  \overline{A\g \um })(t', \yv) | \right ] d\yv, \end{array} $$
which ensures, by the uniform convergence of $({\uv_n(t', \cdot)})_{n \in \N}$ to $\uv(t', \cdot)$ in  $H^m(\R^3)^3$, that for any 
fixed $(t, \x) \in [0, \tau] \times \R^3$, any $t' \in [0, t[$, 
$$ v_n(t, \x; t') \to v(t, \x; t') \quad \hbox{as} \quad n \to \infty. $$
Moreover, the inequality
\BEQ  | v_n(t, \x; t') | \le  \left ( V_{n,0} (t')^2 +|| N_A ||_\infty V_{n,1}(t')   \right )  || \g {\bf T}(t', \cdot) ||_{0, 1}, \EEQ
gives by (\ref{eq:calcul_integral}) and (\ref{eq:Vmnt}), 
$$ | v_n(t, \x; t') | \le { C \E^{-3/2}  [ W(0) + \E^{-1}  || N_A ||_\infty   ]  \over \sqrt {\nu (t-t')} } \in L^1([0, t]). $$
Therefore, Lebesgue's Theorem applies and we have 
$$ \int_0^t v_n(t, \x; t') dt' \to  \int_0^t v(t, \x; t') dt'  \qmbx{as} n \to \infty,$$
in other words, $\uv$ satisfies the integral relation (\ref{eq:oseen_rep}) and the regularity results of Section \ref{eq:Energy_balance} apply for $\uv$. By the same proof as that of Lemma 
\ref{lem:reg_un}, we see that there exists a scalar field $p$ such that 
$(\uv, p)$ is a regular solution to the  regularized NSE (\ref{eq:NSE_eps_intro}), over the time interval $[0, \tau_\E^{(m)}(W(0))]$. 
The transition from local to global time is like in  Theorem \ref{thm:regularity_result}'s proof, by the decrase of the function $x \to \tau_\E^{(m)}(x)$, the time continuity of the trajectories in $H^m$ and the energy balance. The proof of  Theorem \ref{thm:exist_approch} is now completed.  \hfill $\square$

\subsection{Behavior at infinity}
In order to take the limit in the regularized NSE when $\E \to 0$, we need to know how the kinetic energy of the velocity field $\uv$ behaves for large values of $| \x |$. From now, we assume that $A$ is of compact support uniformly in $t$, which means 
that there exists $R_0$ verifying: 
\BEQ \label{eq:support_compact} \forall \, \x \in \R^3 \,\,  \hbox{s.t.} \,\,  | \x | \ge R_0, \quad \forall \, t \ge 0, \quad A(t, \x) = 0. \EEQ
We prove in this subsection: 
\BL \label{lem:alinfini}  There exists a non increasing continous function of $t$, 
$\varphi = \varphi(t)$, such that for any $R_1>0$, $R_2>0$, $R_1 < R_2$, 
\BEQ  \label{eq:infinity} {1 \over 2} \int_{ | \x | \ge R_2}  | \uv (t, \x) |^2 d \x \le {1 \over 2} \int_{ | \x | \ge R_1}  | \uv_0(\x) |^2 d \x  +{1 \over R_2 - R_1} \varphi(t).    \EEQ
\EL 
\begin{proof}
Let $R_1$, $R_2$, $0 < R_1 < R_2$, and $f= f(\x)$ be the function defined by 
\BEQ  \begin{array}{ll} f(\x) = 0 & \hbox{if } | \x | \le R_1, \\
\zoom f(\x) = { | \x | - R_1 \over R_2 - R_1 } & \hbox{if }  R_1 \le | \x | \le R_2, \\
f(\x) = 1 & \hbox{if } | \x | \ge R_2. 
\end{array} 
\EEQ
Taking $f(\x) \uv(t, \x)$ as test in (\ref{eq:NSE_eps}.i) and integrating by parts by using $\div \, \uv = 0$,  yields at each time $t$, 
\BEQ \left \{ \begin{array}{l} \zoom  {1 \over 2} \inte f(\x) | \uv (t, \x) |^2 d \x + \nu \int_0^t \inte f(\x) | \g \uv (t', \x) |^2 d\x dt' = \\~ \\
\zoom {1 \over 2} \inte f(\x) | \overline{\uv_0} (\x) |^2 d \x -  \nu \int_0^t  \inte \g \uv (t', \x) \cdot \g f (\x) \cdot \uv (t', \x) \, d\x dt' + \\ ~ \\
\zoom \int_0^t  \inte \uv(t', \x) \cdot \g f(\x) \, p(t', \x) d\x dt' + {1 \over 2} \int_0^t  \inte  \um(t', \x) \cdot \g f(\x) \, | \uv(t', \x) |^2 d\x dt' - \\~ \\
\zoom - \int_0^t \inte A(t', \x) \g \um(t', \x) : \g (\overline{ f(\x) \uv(t', \x)) } d \x dt'
\end{array} \right. 
\EEQ
Taking $R_1 \ge R_0$ where $R_0$ is specified in (\ref{eq:support_compact}), leads to  
\BEQ  \label{eq:support_compact}  \int_0^t \inte A(t', \x) \g \um(t', \x) : \g (\overline{ f(\x) \uv(t', \x)) } d \x dt' = 0. \EEQ
From there, the calculations carried out in Leray \cite{Ler1934}, section 27, pages 232-234, can be recycled turnkey, and the conclusion follows. 
\end{proof}  
\begin{Remark} The assumption "$A$ is with compact support" is consistent with the idea that no turbulence occurs for large values of $| \x |$, which is in agreement 
with the results of Caffarelli-Kohn-Nirenberg \cite{CKN1982} about the singularitie's location  (if any)  of  NSE's solutions without eddy viscosity, so far we believe that 
turbulence and singularities are connected.  We conjecture that this is not needed, which is leaved as an open question.  
\end{Remark}

\section{Passing to the limit in the equations}\label{sec:passing_to_the_limit}

\subsection{Aim} 

from now, $\E>0$ being given,  $(\uv_\E, p_\E)$ denotes the solution to the regularized NSE (\ref{eq:NSE_eps}). The aim of this section is to show that we can extract from $(\uv_\E, p_\E)_{\E > 0}$
 a subsequence that converges to a turbulent solution of the NSE (\ref{eq:NSE}) (see Definition \ref{def:turb_sols}), which will prove Theorem \ref{thm:turb_exist}. 
 We follow roughly speaking the frame set out by J. Leray to pass to the limit. We have filled in the blanks,  refreshed and customized this frame by using modern tools of analysis, taking into account the eddy viscosity term that is not in Leray's work. 
 \smallskip
 
 Recall that the assumptions about $A$ are:
\begin{enumerate}[i)] 
\item $A \ge 0$ a.e in $\R_+ \times \R^3$,  
\item $A \in C_b(\R_+, W^{1, \infty}( \R^3))$, 
\item $A$ is with compact support in space uniformly in $t$, which means that there exists $R_0 >0$, such that 
$\forall \, t \ge 0$, $\forall \, \x \in \R^3$ s.t. $| \x | \ge R_0$, $A(t, \x) = 0$. 
\end{enumerate} 
We recall that the space of test vector fields $\hbox{E}_\sigma$ we are considering is given by
\BEQ \label{eq:definition_E_sigma} \begin{array}{l} \zoom \hbox{E}_\sigma = \left \{ \wv \in L^1_{loc} (\R_+, H^3(\R^3)^3) \quad \hbox{s.t.}  \quad \wv \in C(\R^+, L^2(\R^3)^3),  
\phantom{\p \wv \over \p t} 
\right. \\  \hskip 2cm  \zoom \left.
 \, \g \wv \in L^\infty (\R, C_b(\R^3)^3), \quad {\p \wv \over \p t} \in L^1_{loc} (\R_+, L^2(\R^3)^3), \quad \div \, \wv = 0 \right \}. 
 \end{array} \EEQ
 Let $\wv \in \hbox{E}_\sigma$. We form the scalar product of $\wv$ with the momentum equation (\ref{eq:NSE_eps}.i) and integrate by parts, which is legitimate  by the regularities of 
$\uv_\E$, $p_\E$ and $\wv$.  We get: 
\BEQ \label{eq:weak_form_epsilon} \left \{  \begin{array}{l} \zoom \inte \overline{\uv_0}(\x) \cdot \wv (0, \x) d\x  =   \inte \uv_\E (t, \x) \cdot \wv (t, \x) d\x  \\~ \\  \zoom 
-  \int_0^t \inte \uv_\E(t', \x) \cdot \left [ \nu \Delta \wv (t', \x) + \div (  \overline{A  \g \overline \wv}) (t', \x) ) + {\p \wv  \over \p t' } (t', \x) \right ]  d\x dt' \\~ \\ \zoom  + \int_0^t \inte [\overline{\uv_\E} (t', \x)  \otimes \uv_\E (t', \x)] : \g \wv(t', \x) \, d\x dt' ,  
\end{array}  \right. 
\EEQ
where also have used:
$$ \begin{array}{l} \zoom \inte \div (\overline {A \g \um}) \cdot \wv = -\inte  \overline {A \g \um} : \g \wv = -\inte A \g \um : \overline \wv = \\
\zoom  \hskip 8cm -\inte  \g \uv : \overline {A \overline \wv }= \inte  \uv \cdot \div ( \overline {A \overline \wv }). \end{array} $$
The goal is to study how to pass to the limit in  (\ref{eq:weak_form_epsilon}) when $\E \to 0$. We note that  
the only avaible estimates which do not depend on $\E$ are those given by the energy balance (\ref{eq:energy_balance}), which shows that the sequence 
$(\uv_\E)_{\E>0}$ is bounded in $L^\infty(\R_+, L^2(\R^3)^3)$ as well as in $L^2(\R_+, H^1(\R^3)^3$. A little bit  more can be said:
\begin{lemma}\label{lem:w(t)} 
\begin{enumerate}[i)] 
\item  The function $t \to W_\E(t)$ is uniformly bounded, namely $\forall \, \E >0$,  $\forall \, t \ge 0$, 
\BEQ \label{eq:inquwepsilon} W_\E(t) \le W_\E (0) \le W(0), \EEQ 
\item For any $\E >0$, $t \to W_\E(t)$ is a non increasing function of $t$.
\end{enumerate}
\end{lemma}
 It results from Helly's Theorem (see for instance in \cite{PLL84}):
\begin{corollary} \label{cor:wtilde} There exists $(\E_n)_{n \in \N}$ that goes to zero when $n \to \infty$, a non increasing function 
$\widetilde W(t)$ such that for all $t \ge 0$, $W_{\E_n}(t) \to \widetilde W(t)$ as $n \to \infty$. 
\end{corollary} 
From now we will consider this sequence $(\E_n)_{n \in \N}$.  
Let us write the identity (\ref{eq:weak_form_epsilon}) under the form 
 \BEQ \label{eq:compactness} \inte \overline{\uv_0}(\x) \cdot \wv (0, \x) d\x  = \inte \uv_{\E_n}  (t, \x) \cdot \wv (t, \x) d\x + I_{1, \E_n}(t, \wv)  + I_{2, \E_n}(t, \wv). \EEQ
We will prove in this section the following results, which will complete the proof of Theorem \ref{thm:turb_exist}.
\BL \label{lem:conv_1} There exists a nondecreasing sequence $(n_j)_{j\in\N}$ such that  for all $t \ge 0$, for all $\wv \in \hbox{E}_\sigma$, the sequences 
$(I_{1, \E_{n_j}}(t, \wv))_{j \in \N}$ and  $(I_{2, \E_{n_j}}(t, \wv))_{j \in \N}$ are convergent sequences. \EL
\BL  \label{lem:conv_2} For  each fixed time $t$, there exists $\uv(t, \cdot) \in L^2(\R^3)^3$
such that $(\uv_{\E_{n_j}}(t, \cdot))_{n \in \N}$ weakly converges to $\uv$ in $L^2(\R^3)^3$. \EL
\BL  \label{lem:conv_3} There exists a set $A \subset \R_+$ the complementary of which is a zero measure set and such that for all $t \in A$, the sequence $(\uv_{\E_{n_j}}(t, \cdot))_{n \in \N}$ strongly converges to $\uv(t, \cdot)$ in $L^2(\R^3)^3$. 
\EL
\BL \label{lem:conv_4} The field $\uv = \uv(t, \x)$ is a turbulent solution to the NSE (\ref{eq:NSE}). 
\EL
This program is divided into two subsections. The first is devoted to prove Lemma  \ref{lem:conv_1}, divided in turn into two sub-subsections, 
one considering $(I_{1, \E_{n_j}}(t, \wv))_{j \in \N}$, the other  $(I_{2, \E_{n_j}}(t, \wv))_{j \in \N}$. In the second subsection we prove Lemma 
\ref{lem:conv_2}, \ref{lem:conv_3} and \ref{lem:conv_4} one after another.

 \subsection{Weak convergence and measures}\label{sec:weak_conv} 
  In all what follows, $\wv$ is any given field of $ \hbox{E}_\sigma$, $t \in \R_+$.\subsubsection{Convergence of  $I_{1, \E_n}(t, \wv)$ } \label{sec:intro_u} 
 As $(\uv_{\E_n})_{n \in \N}$ is bounded in 
 $L^\infty(\R_+, L^2(\R^3)^3) = (L^1(\R_+, L^2(\R^3)^3))'$,  we can extract from the sequence $(\uv_{\E_n})_{n \in \N}$ a subsequence which converges 
 to a field $\widetilde \uv = \widetilde  \uv (t, \x) \in  L^\infty(\R_+, L^2(\R^3)^3)$ for the weak $\star$ topology of $L^\infty(\R_+, L^2(\R^3)^3)$. We still denote this subsequence 
 $(\uv_{\E_n})_{n \in \N}$ for the simplicity. 
 We show in what follows: 
 \BEQ \label{eq:limite_I_1}  \begin{array}{l}\zoom \lim_{n \to \infty}  I_{1, \E_n} (t, \wv) = \\ \zoom 
\hskip 2cm \int_0^t \inte \widetilde  \uv (t', \x) \cdot \left [ \nu \Delta \wv (t', \x) +  \div ( A  \g  \wv) (t', \x) +  {\p \wv  \over \p t' } (t', \x) \right ]  d\x dt'. \end{array} \EEQ 
 From the definition of $ \hbox{E}_\sigma$ (see  (\ref{eq:definition_E_sigma})), $\Delta \wv$, $\zoom {\p \wv \over \p t}$ $\in L^1(\R_+, L^2(\R^3)^3)$, thereby
 \BEQ \label{eq:limite_1} \begin{array}{l} \zoom \lim_{n \to \infty} \int_0^t \inte \uv_{\E_n} (t', \x) \cdot \left [ \nu \Delta \wv (t', \x) +  {\p \wv  \over \p t' } (t', \x) \right ]  d\x dt' 
 = \\~\\
\hskip 4cm  \zoom \int_0^t \inte \widetilde \uv (t', \x) \cdot \left [ \nu \Delta \wv (t', \x) +  {\p \wv  \over \p t' } (t', \x) \right ]  d\x dt'.  \end{array} \EEQ
It remains to   pass to the limit in the term 
$$ \int_0^t \inte \uv_{\E_n} (t', \x)  \,  \div (  \overline{A  \g \overline \wv}) (t', \x)  d \x dt' . $$
Let  $\Delta_{\E_n}$ denotes the difference
$$\Delta_{\E_n}  = \zoom \int_0^t \inte \uv_{\E_n} (t', \x)  \,  \div (  \overline{A  \g \overline \wv}) (t', \x)  d \x dt' -  \int_0^t \inte \widetilde  \uv(t', \x)  \,  \div ( A  \g  \wv) (t', \x)  d \x dt', $$
that we split as
$$\begin{array}{lll}  \zoom   
 \Delta_{\E_n}  & = &  \zoom \int_0^t \inte (\uv_{\E_n} (t', \x) - \widetilde  \uv(t', \x) )  \,  \div ( A  \g  \wv) (t', \x) d \x dt'  + \\~ \\
& & + \zoom \int_0^t \inte (\uv_{\E_n} (t', \x) (  \div (  \overline{A  \g \overline \wv}) (t', \x) -  \div ( A  \g  \wv) (t', \x) )  d \x dt' \\~ \\
& = &  \Delta_{\E_n, 1} +  \Delta_{\E_n, 2}. 
\end{array} $$
As  $\wv \in \hbox{E}_\sigma$ and $A  \in C_b(\R_+, W^{1, \infty}( \R^3))$, then  $\div ( A  \g  \wv) \in L^1(\R_+, L^2(\R^3)^3)$, leading to
$ \Delta_{\E_n, 1}  \to 0$ as $n \to \infty$. Moreover,  the Cauchy-Schwarz inequality and the energy balance yield
$$ | \Delta_{\E_n, 2}  | \le \sqrt {W(0)} \int_0^t  || \div (  \overline{A  \g \overline \wv}) (t', \cdot) -  \div ( A  \g  \wv) (t', \cdot) ) ||_{0,2} dt'. $$
By (\ref{eq:conv_ineq_2}) and algebraic calculations, we get
$$ || \div (  \overline{A  \g \overline \wv}) (t', \cdot) -  \div ( A  \g  \wv) (t', \cdot) ) ||_{0,2} \le  C\E_n || A ||_{1, \infty} || \wv ||_{3, 2}, $$
 hence\footnote{This is where we need $\g A \in L^\infty(\R_+ \times \R^3)$. It is likely that there is a way to do without this hypothesis.}  $ \Delta_{\E_n, 2}  \to 0$ as $n \to \infty$, again because $A \in C_b(\R_+, W^{1, \infty}(\R^3))$ and $\wv \in L^1_{loc}(\R_+, H^3(\R^3)^3)$. In conclusion, we obtain
\BEQ \begin{array}{l} \zoom \lim_{n \to \infty}  \int_0^t \inte \uv_{\E_n} (t', \x)  \,  \div (  \overline{A  \g \overline \wv}) (t', \x)  d \x dt' = \\
\hskip 5cm  \zoom  \int_0^t \inte \widetilde  \uv(t', \x)  \,  \div ( A  \g  \wv) (t', \x)  d \x dt' ,\end{array} \EEQ
hence (\ref{eq:limite_I_1}).  In the following, we shall denote by $I_1(t,\wv)$ the limit of $(I_{1, \E_n}(t, \wv))_{n \in \N}$.  

\subsubsection{Convergence of $I_{2, \E_n}(t, \wv)$}
We  show below  step by step that $I_{2, \E_n}(t, \wv)$ converges to a limit denoted by $I_2(t, \wv)$, which is not well  identified  at this stage. 

{\sl a) The Convective measures } are defined by 
\BEQ \mu_n= \overline{\uv_{\E_n}} \otimes \uv_{\E_n}. \EEQ
Let us fix a given time $T>0$. We will study the sequence $( \mu_n)_{n \in \N}$ on $[0,T]$ for technical conveniences, which is not restrictive since $T$ may be chosen as large as we want. 
The energy balance yields 
$$ \forall \, t' \in [0, T], \quad \inte  \mu_n(t', \x) d \x dt' \le  W(0), $$
 hence $( \mu_n)_{n \in \N}$ is bounded in $L^\infty([0,T], L^1(\R^3)^9)$, which suggests to treat each component of $\mu_n$ as measures. 
 \medskip 
 
{\sl b) Compactness.}  Let $k \in \N$, $B_k = B({\rm O},k) \subset \R^3$ be the ball centered at the origin ${\rm O}$ with radius $k$.  We denote by $M(B_k)$ the set of  radon measures  over $B_k$. Therefore, the considerations above show that  any $k$ being fixed, 
 $$ \hbox{the sequence} \quad ( \mu_n)_{n \in \N} \quad \hbox{is bounded in} \quad  
L^\infty([0,T], M(B_k )^9) = (L^1 ( [0,T], C(B_k)^9))'. $$  
 Let $\mu_n^{k}$ denotes the restriction of $\mu_n$ to the ball $B_k$, 
\BEQ \label{eq:restriction}  \mu_n^{k} = \mu_n \vert_{B_k}. \EEQ 
We deduce from the Banach-Aloaglu theorem combined with the Cantor diagonal argument that there exists a sequence $(n_j)_{j \in \N}$ such that 
each sequence of measures $(\mu_{n_j}^{k})_{j \in \N}$ converges to a measure $\mu^k$ in $L^\infty([0,T], M(B_k )^9)$ weak $\star$. Moreover,  
$$ \forall \, k < k', \quad \mu^{k'} \vert_{B_k} = \mu^k. $$ 
{\sl c) Passing to the limit.}
We show the convergence of  the sequence $(I_{2, \E_{n_j}}(t, \wv))_{j \in \N}$ by the Cauchy's criterion in using the estimate (\ref{eq:infinity}) of Lemma \ref{lem:alinfini}, which guaranties that the $\mu_n$'s have low mass distributions  at infinity. 
Let ${\bf a} = {\bf a}(t, \x) \in L^1 ( [0,T], C_b(\R^3)^9)$, where $C_b(\R^3)$ denotes the space of bounded continuous functions on $\R^3$, $t \in [0, T]$. 
Let 
$$ \lambda_n (t,  {\bf a}  ) = \int_0^t \inte \mu_n (t' \x) : {\bf a}(t', \x)\, d \x dt', $$
$\eta >0$, and $k \in \N$, the choice of which will be decided later.  We have 
\BEQ \label{eq:autre_ineg} \begin{array}{l} \zoom  | \lambda_{n_p} ( t, {\bf a}  ) - \lambda_{n_q} (t,  {\bf a}  ) | \le    \int_0^t || {\bf a}(t', \cdot) ||_{0, \infty} \left ( \int_{ | \x | \ge k} | \mu_{n_p } (t', \x) - \mu_{n_q } (t', \x) | d\x \right ) dt'  \\~\\
\hskip 4cm \zoom + \left   \vert \int_0^t \int_{B_k}  (\mu_{n_p } (t', \x) - \mu_{n_q } (t', \x)) : {\bf a}(t', \x)\, d \x dt'  \right \vert.
\end{array} 
\EEQ
The function $\varphi$ involved in (\ref{eq:infinity}) being non increasing, we get by (\ref{eq:infinity}) for $k \ge 2 R_0$, 
\BEQ \label{eq:inegalite_passage} \begin{array}{l} \zoom \int_0^t || {\bf a}(t', \cdot) ||_{0, \infty} \left ( \int_{ | \x | \ge k} | \mu_{n_p } (t', \x) - \mu_{n_q } (t', \x) | d\x \right ) dt' \le \\ ~ \\ 
\zoom \hskip 3cm \left ( \int_{ | \x | \ge k/2} | \uv_0 (\x) |^2 d\x +{4 \varphi(T) \over k } \right )  \int_0^T || {\bf a}(t', \cdot) ||_{0, \infty} dt'. 
\end{array} 
\EEQ 
As $\uv_0 \in L^2(\R^3)$, we can fix $k \ge R_0$ such that the r.h.s of (\ref{eq:inegalite_passage}) is less than $\eta/2$. Furthermore, for this $k$ and by the definition 
(\ref{eq:restriction}) of the  $\mu_n^k$'s:
$$  \begin{array}{l} \zoom  \left \vert \int_0^t \int_{B_k}  (\mu_{n_p } (t', \x) - \mu_{n_q } (t, \x)) : {\bf a}(t', \x)\, d \x dt'   \right \vert  = \\ ~ \\
\hskip 3cm \zoom \left \vert  \int_0^t \int_{B_k}  (\mu_{n_p }^k (t', \x) - \mu_{n_q }^k (t', \x)) : {\bf a}(t', \x)\, d \x dt' \right \vert = J_{p,q}^k .  \end{array} $$
From the weak convergence of the sequence $(\mu_{n_j}^{k})_{j \in \N}$, we deduce that there exists $j_0$ (depending on ${\bf a}$) such that for any 
$p, q \ge j_0$, we have $J_{p,q}^k \le \eta/2$. Therefore, (\ref{eq:autre_ineg}) gives for such $p,q$, 
$$ | \lambda_{n_p} (t, {\bf a}  ) - \lambda_{n_q} (t, {\bf a}  ) | \le \eta.$$
Then the sequence $( \lambda_{n_j} (t, {\bf a}))_{j \in \N}$  is a Cauchy sequence in $\R$, thus convergent. Let 
$\lambda ( t, {\bf a})$ denotes its limit. In view of the choice of the space $\hbox{E}_\sigma$, for any $t >0$ and any $\wv \in \hbox{E}_\sigma$, $\g \wv \in L^1 ( [0,t], C_b(\R^3)^9)$. Therefore 
\BEQ \label{eq:limit_convective_term} \lim_{j \to \infty} I_{2, \E_{n_j}}(t, \wv)  = \lambda(t, \g \wv).  \EEQ
This concludes the proof of Lemma \ref{lem:conv_1}. In the following, we set $I_2(t, \wv) = \lambda(t, \g \wv)$. \hfill $\square$ 
\smallskip

According to  the equality (\ref{eq:compactness}), the lemma \ref{lem:conv_1} admits the following corollary.
\begin{corollary} \label{cor:convergence}  For all 
$\wv \in \hbox{E}_\sigma$, all $t \in \R_+$, the sequence
$$\left (  \inte \uv_{\E_{n_j} }  (t, \x) \cdot \wv (t, \x) d\x \right )_{j \in \N}$$
has a limit uniquelly determined, namely
\BEQ \label{eq:limite_uniquelly}  \lim_{j \to \infty} \inte \uv_{\E_{n_j} }  (t, \x) \cdot \wv (t, \x) d\x =  \inte \uv_0(\x) \cdot \wv (0, \x) d\x - I_1(t,\wv) -  I_2( t, \wv). \EEQ
\end{corollary}  
\begin{Remark} It would not be surprising that what is done above has something to do with the Young measures (see in \cite{JMB89, LT79, LCY69}). 
There also could be connections with the $H$-measures, initially introduced by L. Tartar (see \cite{LT85, LT90}) as well as with the work by 
A. Majda and R. DiPerna \cite{DM87}. All of this remains to be clarified. 
\end{Remark} 

\subsection{Transition from weak to strong convergence: conclusion } \label{sec:transition_weak_strong} 

{\sl  Proof of Lemma \ref{lem:conv_2}.} Consider 
$${\bf a} \in \bigcap_{m \ge 0} H^m(\R^3)^3, $$
such that $\div \, {\bf a} = 0$. 
Let $t >0$, $\varphi = \varphi(t')$ be a non negative function of class $C^\infty$ less than 1 such that $\varphi(t') = 1$ when $t' \in [0, t+1]$, $\varphi(t') = 0$ 
when $t' \in [t+2, \infty[$. It is easely checked that 
$$ {\bf w} (t', \x) = \varphi(t' ) {\bf a}(\x) \in  \hbox{E}_\sigma. $$
As $ {\bf w} (t, \x) = {\bf a} (\x)$, Corollary \ref{cor:convergence} implies that the sequence
$$\left ( \inte \uv_{\E_{n_j}}  (t, \x) \cdot {\bf a}( \x) d\x \right )_{j \in \N} $$
has a limit uniquely determined. We recall that  $|| \uv_{\E_{n_j}} (t, \cdot) ||_{0, 2} \le \sqrt{W(0)}$.  By Lemma 7 page 209 in \cite{Ler1934}, we can conclude that the sequence  $(\uv_{\E_{n_j}}  (t, \cdot))_{j \in \N}$ has a weak limit in $L^2(\Om)^2$, denoted by $\uv(t, \cdot)$. 
This weak convergence also  leads to: 
\BEQ \label{eq:wtildemaj}  \inte | \uv(t, \x) |^2 d \x \le \widetilde W(t),\EEQ
where $\widetilde W$ is introduced in Corollary \ref{cor:wtilde}.  \hfill $\square$
\smallskip

For the simplicity, we write from now $\E$ instead of 
 $\E_{n_j}$, $\E \to 0$ instead of $j \to \infty$. 
 \medskip

{\sl  Proof of Lemma \ref{lem:conv_3}.} We recall the energy balance satisfied by $\uv_\E$
\BEQ \label{eq:energy_balance_epsilon}  {1 \over 2} W_\E(t) + \nu \int_0^t J_\E^2(t') dt' + \int_0^t  K_{A,\E}^2(t') dt' = {1 \over 2} W_\E(0) \le {1 \over 2} W(0)  , \EEQ 
where 
\BEQ \left \{  \begin{array}{l}  \zoom W_\E(t) = \inte | \uv_\E (t, \x)|^2 d\x, \\
\zoom J_\E^2 (t) = \inte | \g \uv_\E(t, \x)  |^2 d\x, \\
\zoom K_{A,\E}^2(t) = \inte A(t, \x)  | \g \uv_\E(t, \x)  |^2 d\x. \end{array} \right. 
\EEQ

We next deduce from Fatou's Lemma and (\ref{eq:energy_balance_epsilon}),   
\BEQ \label{eq:energy_balance_fin}  \nu \int_0^\infty (\liminf_{\E \to 0} J_\E (t') )^2 dt' \le  {1 \over 2} W(0).  \EEQ
Therefore, $\zoom  t \to \liminf_{\E \to 0} J_\E (t) \in L^1 ([0, \infty[)$. Then  there exists a set $A \subset [0, \infty[$, such that $meas(A^c) = 0$ and 
\BEQ \forall \, t \in A, \quad  
\zoom \liminf_{\E \to 0} J_\E (t) < \infty. \EEQ
Let $ t \in A \cap [0,T]$ be fixed. There is a sequence $(\E_n)_{n\in\N}$ going to zero when $n \to \infty$ (which could depend on $t$) and such that 
 the sequence $(J_{\E_n} (t))_{n \in \N}$ is bounded.  In particular, $(\uv_{\E_n})_{n \in \N}$ is bounded in $H^1(\R^3)$.

Let $\eta >0$ be given. We know from Lemma \ref{lem:alinfini} that there exists $R >0$ (which depends on $T$, $\uv_0 $ and $\eta$) and such that 
\BEQ \label{eq:petituvinfty}  \int_{ | \x | \ge R} | \uv_{\E_n} (t, \x) |^2d \x \le \eta. \EEQ
Furthermore, as the sequence $(\uv_{\E_n} )_{n \in \N}$ is bounded in $H^1(\R^3)$ and has a unique adherence value  in $L^2(B_R)^3$ for the weak topology,   the Rellich-Kondrachov theorem applies: the sequence $(\uv_{\E_n} )_{n \in \N}$ converges to $\uv$ strongly in $L^2(B_R)$. Then (\ref{eq:petituvinfty}) gives
\BEQ \label{eq:limsupinf}  \limsup_{n \to \infty} W_{\E_n} (t) \le \int_{B_R} | \uv (t, \x) |^2 d\x+ \eta \le \inte | \uv (t, \x) |^2 d\x + \eta, \EEQ
which holds for any $\eta>0$. Therefore, combining (\ref{eq:limsupinf}) with the definition of the function 
$\widetilde W$ and the inequality (\ref{eq:wtildemaj}), we get:
\BEQ  \inte | \uv (t, \x) |^2 d\x \le \widetilde W(t) =  \limsup_{n \to \infty} W_{\E_n} (t) \le \inte | \uv (t, \x) |^2 d\x, \EEQ
hence the convergence of $(|| \uv_\E(t, \cdot) ||_{0, 2})_{\E >0}$ to $|| \uv (t, \cdot) ||_{0, 2}$ and the conlusion of the proof, because 
we already know that $(\uv_\E(t, \cdot))_{\E >0}$ weakly converges to $\uv(t, \cdot)$ in $L^2(\R^3)^3$. 
\hfill $\square$

 \medskip
 {\sl Proof of Lemma \ref{lem:conv_4}}.  In order to prove that $\uv$ is a turbulent solution to the NSE, it remains to:
 \begin{enumerate}[a)] 
 \item \label{it:afin} Check that $\widetilde \uv = \uv$, where  $\widetilde \uv$ was introduced in subsection \ref{sec:weak_conv} as the weak star limit 
 of $(\uv_\E)_{\E >0}$ in $L^\infty(\R_+, L^2(\R^3)^3)$, 
 \item \label{it:bfin} Show that 
\BEQ   \label{eq:lim_conv_term}  \forall \, t \in \R_+, \, \forall \, \wv \in \hbox{E}_\sigma, \quad  I_2(t, \wv) = \int_0^t \inte {\uv} (t', \x) \otimes \uv(t', \x) : \g \wv (t, \x) d\x dt',  \EEQ
where $I_2 (t, \wv)$ was defined by
 (\ref{eq:limit_convective_term}),
 \item \label{it:cfin}  Check that $\uv$ satifies the energy inequality, \end{enumerate}
 which is done in the following each item after another.  
 \smallskip
 
 \ref{it:afin}) {\sl Weak star limit identification}. Let ${\bf a} \in  L^1(\R_+, L^2(\R^3)^3)$, and consider 
 $$ \varphi_\E(t) = \inte \uv_\E (t, \x) \cdot {\bf a} (t, \x) d \x. $$
 According to Lemma \ref{lem:conv_3}, 
 $$ \forall \, t \in A, \quad \lim_{\E \to 0} \varphi_\E(t) = \varphi(t) = \inte \uv (t, \x) \cdot {\bf a} (t, \x) d \x.$$
 Moreover, the Cauchy-Schwarz inequality combined with the inequality (\ref{eq:inquwepsilon}) gives 
 $$ |  \varphi_\E(t)  | \le \sqrt{W_\E(t)} || {\bf a} (t, \cdot) ||_{0, 2} \le \en  || {\bf a} (t, \cdot) ||_{0, 2} \in L^1(\R_+).$$
Because $meas(A^C) = 0$, we deduce from Lebesgue Theorem that 
 $$\lim_{\E \to 0}  \int_0^\infty \varphi_\E(t) dt =  \int_0^\infty \varphi(t) dt, $$
 hence $\widetilde \uv = \uv$. 
 
 \smallskip
  \ref{it:bfin}) {\sl Limit in the convective term.} Let $\wv \in \hbox{E}_\sigma$, and 
 $$ \psi_\E (t) = \inte \overline{\uv_\E} (t, \x) \otimes \uv_\E(t, \x) : \g \wv (t, \x) d\x. $$
Let $t \in A$. Obviously $ \overline{\uv_\E} (t, \cdot)  \to \uv (t, \cdot)$ strongly in $L^2(\R^3)^3$. As $\g \wv$ is bounded in space and time, we deduce that for all $t \in A$, 
$$ \psi_\E (t) \to \psi (t) = \inte \uv (t, \x) \otimes \uv(t, \x) : \g \wv (t, \x) d\x \quad \hbox{as} \quad \E \to 0,$$
and by the energy balance
$$ | \psi_\E (t) | \le C W(0) || \g \wv (t, \cdot) ||_\infty \in L^1([0, \infty[). $$
Then, as $meas(A^c) = 0$, we have by Lebesgue's Theorem,
$$ \forall \, t \in \R_+, \quad  \lim_{\E \to 0} \int_0^t  \psi_\E (t')dt' =  \int_0^t  \psi(t')dt',$$
hence (\ref{eq:lim_conv_term}) 

\smallskip

 \ref{it:cfin}) {\sl Energy inequality.} It is easily checked that $\forall \, t \in A$, 
 $\uv_\E(t, \cdot) \to \uv(t, \cdot)$ weakly in $H^1(\R^3)^3$. Therefore, 
 $$ \left ( \inte | \g \uv (t, \x) |^2 d \x \right )^{1/2} =  J(t) \le \liminf_{\E \to 0} J_\E (t),$$
 and as $A \in L^\infty (\R_+ \times \R^3)$ and is non negative, a convexity argument yields
 $$ K_A(t) \le \liminf_{\E \to 0} K_{A, \E}(t). $$
 All these inequalities hold for almost all $t \in A$. Then, by taking the limit in the energy balance (\ref{eq:energy_balance_epsilon}) we get by (\ref{eq:wtildemaj}) and Fatou's Lemma,
\BEQ \label{eq:energy_balance_2} {1 \over 2} W(t) + \nu \int_0^t J^2(t') dt' + \int_0^t K_A(t') dt'  \le {1 \over 2} W(0),\EEQ 
as expected, which finishes the proof. \hfill $\square$

\section {Additional remarks and open questions} 

\subsection{Case $A=0$ and obstruction to generalizations}\label{sec:additional_comments}  

Let us recall one  main Leray's argument when $A=0$, written in our framework. In this case,  any regular solution to the NSE (\ref{eq:NSE}) over the time interval $[0,T[$ satisfies, according 
to (\ref{eq:estimate_oseen}), 
\BEQ \label{eq:v(t)leray} V(t) \le V(0) + C \int_0^t { V^2(t') \over \sqrt{\nu (t-t')} } dt'. \EEQ
where $V(t)$ is defined by (\ref{eq:v_m}). The function  $g(t) = 2 V(0)$ 
 is a supersolution to the non linear Volterra equation, 
$$ f(t) \le V(0) + C \int_0^t { f^2(t') \over \sqrt{\nu (t-t')} } dt', $$
over the time interval 
$I = [0,  4 \nu V^{-2} (0) C^{-1}]$. Therefore, by the V-maximum principle\footnote{To be stringent, we should discuss as in Lemma \ref{lem:local_energy}'s proof to  be in the framework for the application of  the V-maximum principle. At this stage this is not essential and we skip here the details.} 
$$\forall \, t \in I, \quad  V(t) \le 2 V(0). $$
From this, it  is possible to control all the norms of any regular solutions. 
This is why Leray was able to construct a regular solution to the NSE (\ref{eq:NSE}) over $[0, T]$ where $T = O( \nu V^{-2} (0))$, by a fixed point process in a space equipped with the uniform norm in space. He also reported that when 
 a singularty occurs at time $T$, then when $t \to T$ (see section 19 in \cite{Ler1934}), 
$$ V(t) \ge C \sqrt { \nu \over T-t}.$$
The fact that the regular solution can be extended to $[0, \infty[$ when $\nu^{-3}  W(0) V(0)$ is small enough (see Theorem \ref{thm:regular_A=0}) is a tricky 
combination of such arguments, generalized as most as possible, as well as the fact that the turbulent solution becomes regular up to a time 
$O(W(0)^2 / \nu^5)$. 
\smallskip

When $A \not=0$, this does not work anymore, since 
we get, because of the eddy viscosity term, 
\BEQ V(t) \le V(0) + C \int_0^t { V^2(t')+ || N_A ||_\infty V_1 (t' )  \over \sqrt{\nu (t-t')} } dt'. \EEQ 
where $V_1(t') = || \g \uv (t', \cdot) ||_{0, \infty}$. Thereby, to control $V(t)$, we must control $V_1(t)$, which involves $V_2(t)$ and so on, and we do not know how to close 
this sequence of inequalities. This is why Leray's program cannot be recycled turnkey. Ideally, Oseen's work  should be rewritten fort the generalized Stokes problem: 
\BEQ \label{eq:Stokes_A} \left \{ \begin{array}{l} 
\p_t \uv - \nu \Delta \uv - \div( A \g \uv) + \g p = f,  \\ 
\div \, \uv = 0,  
\end{array} \right. 
\EEQ
which is not done already so far we know. 

\medskip
In the case of the approximated system (\ref{eq:NSE_eps}) when $A=0$, inequality (\ref{eq:v(t)leray}) becomes 
\BEQ \label{eq:v(t)leray_epsilon} V(t) \le V(0) + C \E^{-3/2} \en \int_0^t { V(t') \over \sqrt{\nu (t-t')} } dt'. \EEQ
Hence, by the V-maximum principle, 
$$ V(t) \le f(t), $$
where $f(t)$ is the unique continuous solution to the linear Volterra equation defined over $[0, \infty[$, 
$$ f(t) = V(0) + C \E^{-3/2} \en \int_0^t { f(t') \over \sqrt{\nu (t-t')} } dt'.$$
By the theory already developed for the NSE (\ref{eq:NSE}) by J. Leray, no additionnal efforts have to be done to establish the existence of a unique solution 
to the approximated system  (\ref{eq:NSE_eps}), which does not work when $A \not=0$. This is why, everything was to be reconsidered. 
\subsection{Leray-$\alpha$, Bardina and others}

The idea of regularizing the convection term by $\um \cdot \g \uv$ led to the actual concept of Leray-$\alpha$ model, and many other close models (NS-$\alpha$, LANS-$\alpha$, Clark-$\alpha$, NS-Voigt, Bardina...), considered as Large Eddy Simulation models (LES) for simulating turbulent flows, although LES emerged in the 60's with the Smagorinsky's work \cite{JS63}.  

Surprisingly, Leray has planted a seed  that germinates these last two decades in the field of modern LES. See for instance in
Ali \cite{HA12, AH13},  Berselli-Iliescu-Layton \cite{BIL06}, Foias-Holm-Titi \cite{FHT01}, Gibbon-Holm \cite{GHJD06, GH08}, Geurt-Holm \cite{GH06},   llyin-Lunasin-Titi \cite{ILT06}, Layton-Rebholz \cite{LR12}, Rebholz \cite{LR08}, this list being non exhaustive.   

These models are based on a regularization calculated by the Helmlhoz filter determined by:
\BEQ \label{eq:helmholz} 
- \alpha^2 \Delta \overline \psi + \overline \psi = \psi \quad  \hbox{in } \R^3, \EEQ
where in this framework the regularizing parameter is named $\alpha$ instead of $\E$. The models are usually considered with periodic boundary conditions, more rarely in a bounded domain with the no-slip condition. The case of an unbounded domain and/or the full space was  not considered before so far we know.

 In Berselli-Lewandowski 
\cite{BL17}, we have investigated the simplified Bardina model in the whole space. Initially introduced by Bardina-Ferziger- Reynolds \cite{BFR80} for weather forecast, this model was 
studied in  \cite{CLT06, LL03, LL06} in the case of periodic boundary conditions. In its simplified version, this model is given by the system 
\BEQ \label{eq:bardina}  \left \{\begin{array}{ll} \zoom 
\p_t \uv+  \g \cdot (\overline{\uv \otimes \uv}) - \nu \Delta \uv + \g p = 0 & \hbox{in } \R^3, \\
\g \cdot \uv=  0 \phantom{\int_0^N}  & \hbox{in } \R^3, \\
\uv_{ t = 0} =\overline{ \uv_0},
\end{array} \right.  
\EEQ 
in which the bar operator is specified by the Helmholz filter (\ref{eq:helmholz}). 
We prove in \cite{BL17} the existence of a unique regular solution to (\ref{eq:bardina}), global in time, that converges to a turbulent solution to the NSE. Attention must be paided with the initial data and the meaning of ``regular solution", since the regularizing effect of the Helmholz filter is lower than that given by a molifier.

What is done in  
 \cite{BL17} is in the same spirit as what is done in the present paper, inspired by Leray'work. It remains to proceed to the same analysis for  the other LES models of this 
 $\alpha$-class mentioned above.
\subsection{Towards the NS-TKE model}The result of Theorem  \ref{thm:turb_exist} still holds when $A \in L^\infty(\R_+, L^\infty( \R^3))$. Indeed, we can approach  $A$ by a sequence $(A_\E)_{\E >0}$, where  
$A_\E \in C_b(\R_+, W^{1, \infty}( \R^3))$, and then pass to the limit in the formulation (\ref{eq:weak_form_epsilon}) when $\E \to 0$. We also can consider the NSE (\ref{eq:NSE}) with a source term 
${\bf f}$ that satisfies a suitable decay condition at infinity. However, we lose the benefit  provided by the monotocity of the function $t \to W(t)$ and we must find out what is the right  function to be considered to replace $t \to \widetilde W(t)$ introduced in Corollary 
\ref{cor:wtilde}. It is not clear that the best choice is $t \to \limsup_{ \E \to 0} W(t)$. This point remains to be discussed, though it is not intractable.  
We have not considered  these issues  to avoid making the text more cumbersome.  
\smallskip

However, this is the right track to tackle the problem of the NS-TKE model in the whole space:
$$ \label{eq:TKE_system_complet_1}  \left \{  \begin{array}{rcll} 
\zoom \p_t \uv  + \uv \cdot \g \uv -\div \left [ (2\nu+ C_u \ell |k|^{1/2}) \, \g  \uv \right ] + \nabla p & = &  0, \\
 \div \, \uv&=& 0,  \\
\zoom \p_t k + \uv \cdot \g k - \div( (\mu + C_k \ell  |k|^{1/2} ) \g k) &= &  \zoom C_u \ell |k|^{1/2} | \g  \uv |^2  -  \ell^{-1} k |k|^{1/2}, \\
 (\uv, k)_{t=0} & = & (\uv_0, k_0), 
\end{array}   \right. $$ 
which is the basic "Reynolds Averaged Navier Stokes`` model of turbulence (see in \cite{CL14}). In this system, $(\uv, p)$ is the mean flow field, and $k$ the turbulent kinetic energy, that measures 
the intensity of the turbulence in a turbulent flow. The function $(t, \x) \to \ell (t, \x)$ is the Prandlt mixing lenght, at this stage a given non negative function, $C_u$ $C_k$ are experimental constants. 
This problem was initially studied in \cite{RL97} in a bounded domain $\Om \subset \R^3$, with homogeneous boundary conditions. The case of $\R^3$ yields a very hard mathematical problem.

\begin{appendices}

\makeatletter
\renewcommand\theequation{\thesection.\arabic{equation}}
\@addtoreset{equation}{section}
\makeatother

\section{Estimates for the Oseen tensor}\label{app:ossen}This appendice has entirely been written by Paul Alphonse and Adrien Laurent.
\begin{theorem}
For $t\in \R^{+*}$ and $\abs{x}>0$, let $$G(t,x)=\frac{1}{\abs{x}}\int_0^{\abs{x}} \frac{e^{-\frac{\rho^2}{4\nu t}}}{\sqrt{t}} \dd \rho.$$
Let 
$$T_{ii}=-\frac{\partial^2 G}{\partial x_j^2}-\frac{\partial^2 G}{\partial x_k^2} \text{ and } T_{ij}=\frac{\partial^2 G}{\partial x_i\partial x_j}$$
be the Oseen tensor.
Then the followig estimates are verified :
$$\abs{T(t,x)}\leq \frac{C}{(\abs{x}^2+\nu t)^\frac{3}{2}},$$
$$\abs{D^m T(t,x)}\leq \frac{C_m}{(\abs{x}^2+\nu t)^\frac{m+3}{2}}.$$
\end{theorem}

\begin{proof}
The function $G$ can be extended on $\abs{x}=0$ as a $\CC^\infty$ function.
We have $$\frac{\partial G}{\partial x_i}(t,x)=\frac{x_i}{\abs{x}^2}\left(\frac{e^{-\frac{\abs{x}^2}{4\nu t}}}{\sqrt{t}}-G(t,x)\right).$$
Integrating by part $G$ yields
$$G(t,x)=\frac{e^{-\frac{\abs{x}^2}{4\nu t}}}{\sqrt{t}} +\frac{1}{2\nu t^{\frac{3}{2}}\abs{x}}\int_0^{\abs{x}} \rho^2\frac{e^{-\frac{\rho^2}{4\nu t}}}{\sqrt{t}} \dd \rho.$$
Thus $$\frac{\partial G}{\partial x_i}(t,x)=-\frac{x_i}{2\nu t^{\frac{3}{2}}\abs{x}^3}\int_0^{\abs{x}} \rho^2\frac{e^{-\frac{\rho^2}{4\nu t}}}{\sqrt{t}} \dd \rho.$$
With this same trick, one finds
$$\frac{\partial^2 G}{\partial x_i\partial x_j}(t,x)=-\frac{1}{6\nu t^{\frac{3}{2}}} e^{-\frac{\abs{x}^2}{4\nu t}} \delta_{ij}
+\left(\frac{x_i x_j}{4\nu^2 t^{\frac{5}{2}}\abs{x}^5}-\frac{1}{12\nu^2 t^{\frac{5}{2}}\abs{x}^3}\delta_{ij}\right)\int_0^{\abs{x}} \rho^4 e^{-\frac{\rho^2}{4\nu t}} \dd \rho.$$
And by a change of variables,
$$\frac{\partial^2 G}{\partial x_i\partial x_j}(t,x)=-\frac{1}{6\nu t^{\frac{3}{2}}} e^{-\frac{\abs{x}^2}{4\nu t}} \delta_{ij}
+\left(\frac{8\nu^{\frac{1}{2}} x_i x_j}{\abs{x}^5}-\frac{8\nu^{\frac{1}{2}}}{3\abs{x}^3}\delta_{ij}\right)\int_0^{\frac{\abs{x}}{2\sqrt{\nu t}}} \rho^4 e^{-\rho^2} \dd \rho.$$
We then have
$$\abs{T_{ij}}\lesssim
\frac{1}{t^{\frac{3}{2}}} e^{-\frac{\abs{x}^2}{4\nu t}}
+\frac{1}{\abs{x}^3}\int_0^{\frac{\abs{x}}{2\sqrt{\nu t}}} \rho^4 e^{-\rho^2} \dd \rho.$$
Finally, by denoting $y=\frac{\abs{x}}{2\sqrt{\nu t}}$, we have
$$(\abs{x}^2+\nu t)^\frac{3}{2}\abs{T_{ij}}\lesssim
(1+y^2)^{\frac{3}{2}}e^{-y^2}
+(1+\frac{1}{y^2})^\frac{3}{2}\int_0^{y} \rho^4 e^{-\rho^2} \dd \rho.$$

The first term of the inequality is bounded for all $y\in \R^+$.
We then denote $\varphi(y)$ the function corresponding to the second term of the inequality. The function $\varphi$ is continuous on $\R^{+*}$ and verifies
$$\lim\limits_{y \rightarrow +\infty}\varphi(y) =\int_0^{+\infty} \rho^4 e^{-\rho^2} \dd \rho <+\infty.$$
For the case $y\to 0$, we notice by integrals comparison that
$$\int_0^{y} \rho^4 e^{-\rho^2} \dd \rho \underset{y\to 0}{\sim} \frac{y^5}{5}.$$
Then $$\varphi(y)\underset{y\to 0}{\sim}\frac{y^2}{5},$$
and $\varphi$ is bounded on $\R^+$.
This work gives us that
$$\abs{T(t,x)}\leq \frac{C}{(\abs{x}^2+\nu t)^\frac{3}{2}}.$$

For the estimates on the derivatives, one can show by induction that
$$\abs{\frac{\partial^m}{\partial x_{i_1}...\partial x_{i_m}}\frac{\partial^2 G}{\partial x_i\partial x_j}}(t,x)\lesssim
P_m(y)\frac{1}{t^{\frac{m+3}{2}}}e^{-y^2}
+\frac{1}{\abs{x}^{m+3}}\int_0^{y} \rho^{4+2m} e^{-\rho^2} \dd \rho,$$
where $P_m$ is a polynomial of degree $m$.
Adapting the same method as before yields the estimate on $D^m T$.
\end{proof}

\section{Non linear Volterra equations and V-maximum principle}\label{ap:appen_volt} 
The results of this section about the non linear Volterra equations and the V-maximum principle are due to Luc Tartar. 
\subsection{Framework} 
Let $a \in \R_+$, $T \in \R_+^\star$, $K \in L^1([0,T])$, $k \ge 0$ {\sl a.e.} in $[0,T]$, $P$ a continuous non increasing real valued function.  We consider the following 
functional equation,  
\BEQ \label{eq:equation_volterra}  f(t) =  a + \int_0^t K(t-t') P(f(t')) dt'. \EEQ
When $P(f) = f$, this equation is a Volterra equation. As we have seen in this paper, we have to consider $P$ that are not linear, and this 
is why we call this equation a generalized non linear Volterra equation. 
The aim of this appendix is to prove a maximum principle which states that subsolutions of (\ref{eq:equation_volterra})  are below supersolutions. 

In this section, we define 
the notions of sub and super solutions, and we show how to construct solutions from these sub-super solutions. In the following, we set for any 
$f \in L^\infty([0,T])$, $a \ge 0$, 
\BEQ S[a, f](t) = a + \int_0^t K(t-t') P(f(t')) dt', \quad t \in [0,T].  \EEQ
We first note that as $P$ is non increasing and $K \ge 0$, 
when $f \le g$, then $S[a, f] \le S[a, g]$. Moreover, when 
$f \in L^\infty([0,T]$, then $S[a, f] \in C([0,T])$. 
\begin{definition} We say that $f \in L^\infty([0,T])$ is a subsolution of (\ref{eq:equation_volterra}) if $ f \le S[a, f]$. We say that 
$g \in L^\infty([0,T])$ is  a supersolution of (\ref{eq:equation_volterra}) if $ S[a, g] \le g$. 
\end{definition}
\begin{Remark} We remark that $f=0$ is always a subsolution of (\ref{eq:equation_volterra}). However, it is important to note that the solution of (\ref{eq:equation_volterra}) may be not defined over $[0,T]$. Take for instance $K=1$, $P(z)=z^2$. 
Then (\ref{eq:equation_volterra}) becomes the differential equation $f' = f^2$, $f(0)=0$, whose solution is $f(t) = a (1 - at)^{-1}$, which blows up at a time 
less than $T$ when $a T >1$. In such case, there is no supersolution over $[0,T]$. 
\end{Remark} 
As a consequence of the assumption $K \in L^1([0, T])$,  the following result is straightforward.
\BL \label{lem:supersol} Let $G > a$. Then there exists $\tau \in ]0, T]$ such that $g(t) = G$ is a supersolution of (\ref{eq:equation_volterra}) over $[0, \tau]$. 
\EL
Assume now that there exists a supersolution $g \ge 0$ of (\ref{eq:equation_volterra}) over $[0,T]$, and let $(g_n)_{n \in \N}$ be the sequence defined by 
$$ g_0 = g, \quad g_{n+1} = S[a, g_n]. $$
We obvioulsy have $0 \le g_{n+1} \le g_n$ for all $n$, and
\BL \label{lem:supersol} The sequence $(g_n)_{n \in \N}$ uniformly converges to a solution of (\ref{eq:equation_volterra}). \EL 
\begin{proof} We first observe that $g_n$ is continuous for $n \ge 1$. By monotonicity and since $g_n \ge 0$,  $(g_n)_{n \in \N}$ simply converges to some $f^+ \in L^\infty([0,T])$. As for $n \ge 1$
$$ | K(t-t') P( g_n(t')) | \le K(t-t') \max( | P(0) |, | P( \max_{[0, T]} g_1 ) | ) \in L^1 ([0, t]),$$
and $P$ is continous, we deduce from Lebesgue's Theorem that for all $t \in [0,T]$, $S[a, g_n ](t) $ converges to 
$S[a, f^+](t) $.  The inequalities $g_{n+1} \le g_n$ yields that $f^+$ is a solution of (\ref{eq:equation_volterra}), hence $f^+$ is continuous and 
by Dini's Theorem, the convergence of the sequence $(g_n)_{n \in \N}$ is uniform. We also notice that $f^+ \le g_n$ for all $n$. 
\end{proof}

\subsection{Uniqueness} 

The solution to (\ref{eq:equation_volterra}) may be not unique. For instance, take $K=1$ and $P(z) = \sqrt z$, $a=0$. Therefore (\ref{eq:equation_volterra}) becomes the differential equation $f' = \sqrt f$, $f(0) = 0$, whose solutions are $f(t) = 0$ and $f(t) = t^2 /4$. However, when $P$ is Lipchitz, uniqueness occurs in some sense, which is the aim of this section. 

Assume that (\ref{eq:equation_volterra}) has a subsolution $f$ and a supersolution $g$ that verify $f \le g$, both being continuous. 
Arguing as in lemma \ref{lem:supersol},  we see in this case that the sequence $(f_n)_{n \in \N}$ defined by $f_0 = f$, $f_{n+1} = S[a, f_n]$, uniformly converges  
to a solution $f^-$ of (\ref{eq:equation_volterra}), that also satisfies $f^- \le f^+$. 
\begin{Remark} As $0$ is a subsolution, according to Lemma \ref{lem:supersol}, we have shown that there exists $\tau>0$ such that 
the nonlinear Volterra equation (\ref{eq:equation_volterra}) has a solution over $[0,\tau]$ 
\end{Remark}
Our uniqueness result is phrased as follows. 
\BL \label{eq:uniqueness} Assume that $P$ is a non increasing Lipschitz continuous function, $K \in L^1([0,T])$. Then $f^+ = f^-$ over $[0,T]$. \EL
\begin{proof} Let $L$ denotes the Lipchitz constant of $P$. Then we have
\BEQ \label{eq:uniqueness1} \forall \, t \in [0,T], \quad 0 \le f^+(t) - f^-(t) \le L \int_0^t K(t-t') ( f^+(t') - f^-(t') ) dt'.\EEQ 
We first assume that $K$ is bounded by a constant $M$. Therefore,  (\ref{eq:uniqueness1}) yields 
\BEQ \label{eq:uniqueness2} \forall \, t \in [0,T], \quad 0 \le f^+(t) - f^-(t) \le L M \int_0^t ( f^+(t') - f^-(t') ) dt',\EEQ 
from which we easely deduce 
\BEQ \label{eq:uniqueness3} \forall \, t \in [0,T], \, \forall \, m \ge 2, \quad 0 \le f^+(t) - f^-(t) \le {  (L M  t)^m \over {m!} } \sup_{t' \in [0,T]} ( f^+(t') - f^-(t') ),\EEQ 
hence $f^+ = f^-$. For $K \in L^1([0,T])$, we rephrase  (\ref{eq:uniqueness1}) as 
$$ 0 \le \E\le \Phi (\E),$$
by writting $\E = f^+ - f^-$, and 
$$ \Phi (u) (t) = L \int_0^t K(t-t') u(t') dt'. $$
The operator $\Phi$ is a linear operator, the kernel of which is equal to $\widetilde K (t) = K(t) \hbox{1} \! \hbox{I} _{[0,t]}$. The kernel of the operator 
$\Phi^2$ is equal to $\widetilde K \star \widetilde K$, which is continuous, then bounded on the compact $[0,T]$, which yields a similar inequality as 
(\ref{eq:uniqueness3}) and the conclusion $f^+ = f^-$. 
\end{proof} 
\begin{Remark}\label{rem:linear}  With the assumptions of  Lemma \ref{eq:uniqueness}, When $P$ is linear, that is $P(f) = \alpha_1 + \alpha_2 f$ ($\alpha_i \ge 0$), it easy checked that the solution of (\ref{eq:equation_volterra}) can be extended over $[0, \infty[$. In this case, (\ref{eq:equation_volterra})  is referred to as linear Volterra equation. 
\end{Remark} 
\subsection{V-maximum principle} \label{app:section_3} 
The aim of this section is to prove the following result. We still assume $K \in L^1([0,T])$ and that $P$ is a non increasing  Lipchitz-continuous function. 
\BL \label{lem:max_princ}  Let $f$ be a subsolution of  (\ref{eq:equation_volterra}) and $g$ a supersolution of (\ref{eq:equation_volterra}) over $[0,T]$. Then 
\BEQ  \label{eq:B1}  \forall \, t \in [0,T], \quad f(t) \le g(t). \EEQ 
\EL
\begin{proof} 
By considering $S[a, f]$ instead of $f$, we can assume that $f$ is continuous without loss of generality. Similarly, by considering the sequence $(g_n)_{n \in \N}$ as 
in Lemma \ref{lem:supersol}, we can assume that $g$ is a solution of  (\ref{eq:equation_volterra}) instead being a supersolution. 

Assume that (\ref{eq:B1})  do not hold, and let 
\BEQ  \tau = \sup \{ t \in [0,T[, \, \hbox{s.t. }  \, \forall \, t' \in [0,t],  \, f(t') \le g(t') \} \EEQ
Our assumption yiels $\tau <T$ and there exists a sequence $(t_n)_{n \in N}$ that converges to $\tau$, such that 
$t_n >\tau$ for each $n$ and $f(t_n) > g(t_n)$. We may have $\tau=0$. 

Given $\eta>0$, let $k=k(t)$ be the function defined by 
\BEQ k(t) = \left \{ \begin{array}{l} g(t), \quad t \in [0, \tau], \\
g(t) + \eta, \quad t \in \,  ]\tau, T]. \end{array} \right. \EEQ
We claim that there exists $S > \tau$ such that $k$ is a supersolution of (\ref{eq:equation_volterra}) over $[0, S]$ and 
$f \le k$ over $[0, S]$. Indeed, as $P$ is Lipchitz continuous and non increasing, 
\BEQ \forall \, t' \in [\tau, T], \quad 0 \le P(k(t')) - P(g(t')) \le L \eta. \EEQ
Therefore, since $g$ is a solution of (\ref{eq:equation_volterra}),  $k$ is 
a supersolution of (\ref{eq:equation_volterra}) over $[0,S]$ for all $S > \tau$  that satisfy
\BEQ  \label{eq:B5} \forall \, t \in [\tau, S], \quad  L \eta \int_\tau^t K(t-t') dt' \le \eta. \EEQ
As $K \in L^1([0, T])$ there exists $S_0 > \tau$ such that for all $ S \in \,  ]\tau,  S_0]$, (\ref{eq:B5}) holds. Furthermore, as $f$ is continuous, 
there exists $S \in \,  ] \tau, S_0]$ such that 
$$ \forall \, t \in [\tau, S], \quad f(t) \le k(t).$$
We consider the sequences, over $[0, S]$
$$ f_0 = f, \quad f_{n+1} = S[a, f_n], \quad k_0 = k, \quad k_{n+1} = S[a, k_n]. $$
we have over $[0,S]$ and for all $n$, 
$$ f \le f_n \le k_n \le k.$$
According to Lemma \ref{lem:supersol}, $(k_n)_{n \in \N}$ and $(f_n)_{n \in \N}$ converge to a solution  of 
(\ref{eq:equation_volterra}) over $[0,S]$ which is above $f$ over $[0,T]$. By the uniqueness result of 
Lemma \ref{eq:uniqueness}, this solution is the restriction of $h$ to $[0, S]$, which contradicts the definition of $\tau$ and concludes the proof.   
\end{proof}

\end{appendices}

\bibliographystyle{plain}
\bibliography{Leray_Turb_V4}

\end{document}